\documentclass[a4paper,12pt]{amsart}
\usepackage[utf8]{inputenc}
\usepackage[T1]{fontenc}
\usepackage[UKenglish]{babel}
\usepackage[margin=20mm]{geometry}
\usepackage{color}
\usepackage{bbm,bm}

\usepackage{graphicx}

\usepackage{amsmath,amssymb,amsfonts,amsthm}
\usepackage{mathrsfs,eucal,dsfont}
\usepackage{verbatim,enumitem}

\usepackage{hyperref,url}
\usepackage{amssymb, amsmath, amsthm, amscd}

\usepackage{eucal,mathrsfs,dsfont}
\usepackage{color}

\renewcommand{\leq}{\le}
\renewcommand{\geq}{\ge}

\newcommand{\e}{\text{e}}

\def\d{{\rm d}}
\def\<{\langle}
\def\>{\rangle}

\newtheorem{theorem}{Theorem}[section]
\newtheorem{lemma}[theorem]{Lemma}
\newtheorem{proposition}[theorem]{Proposition}

\numberwithin{equation}{section}

\theoremstyle{definition}
\newtheorem{definition}[theorem]{Definition}

\begin{document}
\allowdisplaybreaks
\title[Polynomial lower bound on the effective resistance]
{\bfseries  Polynomial lower bound on the effective resistance for the one-dimensional critical long-range percolation}

\author{Jian Ding \qquad  Zherui Fan  \qquad  Lu-Jing Huang}

\thanks{\emph{J. Ding:}
School of Mathematical Sciences, Peking University, Beijing, China.
  \texttt{dingjian@math.pku.edu.cn}}
\thanks{\emph{Z. Fan:}
School of Mathematical Sciences, Peking University, Beijing, China.
  \texttt{1900010670@pku.edu.cn}}

\thanks{\emph{L.-J. Huang:}
School of Mathematics and Statistics, Fujian Normal University, Fuzhou, China.
  \texttt{huanglj@fjnu.edu.cn}}


\date{}
\maketitle

\begin{abstract}

In this work, we study the critical long-range percolation on $\mathds{Z}$, where an edge connects $i$ and $j$ independently with probability $1-\exp\{-\beta |i-j|^{-2}\}$ for some fixed $\beta>0$. Viewing this as a random electric network where each edge has a unit conductance, we show that with high probability the effective resistances from the origin 0 to $[-N, N]^c$ and from the interval $[-N,N]$ to $[-2N,2N]^c$ (conditioned on no edge joining $[-N,N]$ and $[-2N,2N]^c$) both have a polynomial lower bound in $N$. Our bound holds for all $\beta>0$ and thus rules out a potential phase transition (around $\beta = 1$) which seemed to be a reasonable possibility.


\medskip


\end{abstract}
\allowdisplaybreaks


\section{Introduction}\label{intro}

Consider the critical long-range percolation (LRP) on $\mathds{Z}$, where edges $\langle i,j\rangle $ with $|i-j|=1$ (i.e.\ $i$ and $j$ are nearest neighbors) occur independently with probability $p_1$, while edges $\langle i,j\rangle $ with $|i-j|>1$ (which we refer to as long edges in what follows) occur independently with probability
\begin{equation}\label{LRP}
1-\exp\left\{-\beta\int_i^{i+1}\int_j^{j+1}\frac{1}{|u-v|^2}\d u\d v\right\}.
\end{equation}
Here, $\beta>0$ is a parameter of this LRP model.
We also call this model as a $\beta$-LRP model, and denote $\mathcal{E}$ for the edge set. Additionally, for ease of notation, we will also use $\sim$ to denote edges, that is, $i\sim j$ implies $\langle i,j \rangle \in \mathcal{E}$.
Note that we chose the expression \eqref{LRP} for the connecting probability since it has a strong scaling invariance property and is thus a convenient choice especially when studying scaling limits. That being said, it would be clear that our proof extends to the case when the connecting probability is within a multiplicative constant of $\beta|i-j|^{-2}$.

In this study, we focus on the particular case for $p_1=1$, where the connectivity, thus the existence of the infinite cluster, is trivial.
Placing a unit conductance on each edge, we may then view this LRP model as a random electric network on $\mathds{Z}$.
Our goal is to study the effective resistance (of the $\beta$-LRP), which we now define. Let $f$ be a function defined on $\mathds{Z}^2$. We call $f$ a flow on the $\beta$-LRP model if it satisfies
\begin{equation*}
f_{ij}=
\begin{cases}
-f_{ji},\quad & \langle i,j\rangle\in \mathcal{E}, \\
0,\quad &\text{otherwise}
\end{cases}
\end{equation*}
for all $i,j\in \mathds{Z}$.
So the net flow out of a point $i$ is $f_{(i)}:=\sum_{j\neq i} f_{ij}$.
Moreover, given two finite disjoint subsets $I_1,I_2\subset \mathds{Z}$, we say $f$ is a unit flow from $I_1$ to $I_2$ if it is a flow satisfying
$$
\sum_{i\in I_1}f_{(i)}=1\quad \text{and}\quad f_{(j)}=0\quad \text{for all }j\notin I_1\cup I_2.
$$
Note that the condition above implies $\sum_{i\in I_2}f_{(i)}=-1$.
The effective resistance between disjoint subsets $I_1, I_2\subset \mathds{Z}$ is then defined as
\begin{equation}\label{eff-resis}
R(I_1,I_2)=\inf_{J_1, J_2}\inf\left\{\frac{1}{2}\sum_{i\sim j}f^2_{ij}:\ f\ \text{is a unit flow from $J_1$ to } J_2\right\},
\end{equation}
where the first infimum is taken over all finite subsets $J_1 \subset I_1$ and $J_2 \subset I_2$.
In particular, when $I_1=\{i\}$ is a singleton, we denote $R(i,I_2)$ as $R(I_1,I_2)$.

Our main result establishes polynomial lower bounds on effective resistances, as detailed in the following theorem.
For any subset $A\subset \mathds{Z}^2$, we denote by $\mathcal{E}_A = \mathcal{E} \cap A$.

\begin{theorem}\label{mr}
For any $\beta>0$, there exists a constant $\delta>0$ {\rm(}depending only on $\beta${\rm)} such that the following holds.
For any $\varepsilon\in (0,1/2]$, there exists a constant $c=c(\varepsilon)>0$ {\rm(}depending only on $\beta$ and $\varepsilon${\rm)} such that for all $N\geq 1$,
\begin{equation}\label{point}
\mathds{P}\left[R(0,[-N,N]^c)\geq cN^{\delta}\right]\geq 1-\varepsilon,
\end{equation}
and
\begin{equation}\label{box}
\mathds{P}\left[R([-N,N],[-2N,2N]^c)\geq cN^{\delta} \big|\mathcal{E}_{[-N,N]\times[-2N,2N]^c}=\varnothing \right]\geq 1-\varepsilon.
\end{equation}
\end{theorem}

\subsection{Related work}\label{rw}

Electric networks are commonly associated with reversible Markov chains,
providing a sophisticated and efficient method for understanding properties of these chains \cite{DS84,LP06}.
An essential measurement in electric network theory is the effective resistance, which plays a crucial role in evaluating the conductivity of the electric network. Effective resistances are closely related to various aspects of random walks on networks, including recurrence/transience, heat kernels and mixing times, see e.g. \cite{AF02,LPW09,K14}.

The study of effective resistances for LRP on $\mathds{Z}^d$ has sparked great interests, as this measurement is not only  an intriguing intrinsic property of the underlying random graph but also provides effective tools for studying random walks on this model.
Specifically, consider the sequence $\{p_x\}_{x\in \mathds{Z}^d}$, where $p_x\in [0,1]$ and $p_x=p_{x'}$ for all $x,x'\in \mathds{Z}^d$ such that $|x|=|x'|$. Assume that
$$
0<\beta :=\lim_{|x|\rightarrow \infty}p_x|x|^{s}<\infty
$$
for some $s>0$.
The LRP on $\mathds{Z}^d$, introduced by \cite{S83,ZPL83}, is defined by edges $\langle x,y\rangle $ occurring independently with probability $p_{x-y}$.

We first review progress on effective resistances as well as behavior for random walks on non-critical (i.e.\ $(d,s)\neq (d,2d)$) LRP models.
In \cite{KM08} the authors employed methods to estimate volumes and effective resistances from points to boxes, obtaining the corresponding heat kernel estimates for the case where $d=1,\ s>2$ and $p_1=1$. In fact, the authors showed this in a more general random media setting.
In \cite{M08}, the author derived up-to-constant estimates for the box-to-box resistances for $d=1,\ s\in (1, 2) \cup (2, \infty)$ as well as for $d \geq 2$ and $s\in (d, d+2) \cup (d +2, \infty)$.
For random walks on infinite clusters of these LRP models,
when $s\in (d,d+2)$, it was shown in \cite{CS13} (for $s\in (d,d+1)$) and \cite{BT24} (for $d\geq 2$ and $s\in [d+1,d+2)$) that the random walk on the infinite cluster converges to an $\alpha$-stable process with $\alpha=s-d$. In the case where $d=1$ and $s>2$, it was also shown in \cite{CS13} that the random walk converges to a Brownian motion.
It is worth mentioning that \cite{BCKW21} studied the quenched invariance principle for random walks on LRP graphs with $s>2d$ for all $d\geq 2$. 
According to \cite[Problem 2.9]{BCKW21}, it seems that it remains a challenge to establish scaling limits of random walks on the LRP models with $d\geq 2$ and $s\in (d+2, 2d]$.
There are also numerous related results regarding the heat kernel, mixing time and local central limit theorem of the random walk, see e.g. \cite{BBY08,CS12,CCK22,CKW24}.


There are relatively few results for effective resistances and random walks on the critical LRP model for $d\in \{1,2\}$ (note that in \cite{M08} resistance bounds are obtained in the critical case for $d\geq 3$).
The author of \cite{B02} established recurrence for the random walk in the case of $s=2d$ for $d\in \{1,2\}$, by showing that the effective resistance from the origin to $[-N,N]^c$ diverges as $N \to \infty$.
This was then extended in \cite{Baum23} to more general LRP models with weight distribution satisfying some moment assumptions. In addition,
bounds on box-to-box resistances were provided in \cite{M08} for $d\in \{1,2\}$, including a constant lower bound for $d = 1$.
Despite these progresses, it seems a challenge to determine the divergence rate of the effective resistance.
In fact, from our conversations with colleagues, it seems there is a folklore debate on whether for $d=1$ and $s=2$ the effective resistance grows polynomially for all $\beta>0$ or has a phase transition at $\beta=1$. Our contribution in this work confirms the former scenario.


\subsection{Outline of the proof of Theorem \ref{mr}}\label{outline}

Since the main ideas for proving \eqref{point} and \eqref{box} in Theorem \ref{mr} are essentially same, we mainly  offer an overview for the proof of \eqref{point}.
To begin with, let us fix a sufficiently large $N\in \mathds{N}$ and recall that the effective resistance $R(0,[-N,N]^c)$ is defined in \eqref{eff-resis} with $I_1=\{0\}$ and $I_2=[-N,N]^c$. Our objective is to establish a polynomial lower bound for the effective resistance $R(0,[-N,N]^c)$.

When $\beta\in (0,1)$, intuitively there are relatively few long edges. This actually allows for a straightforward proof of a polynomial lower bound on $R(0,[-N,N]^c)$ by finding a sufficient number of cut-edges. Here, an edge is a cut-edge (and in general a set is a cut-set) if 0 is disconnected from $[-N, N]^c$ after its removal.
To be more specific, for $i\in [-N,N)_{\mathds{Z}}$, define
$$
\xi_i=
\begin{cases}
1,\quad &(i,i+1)\ \text{is a cut-edge},\\
0,\quad &\text{otherwise}.
\end{cases}
$$
A simple calculation reveals that $\mathds{P}[\xi_i=1]=(2N)^{-\beta}$, leading to  $$
\mathds{E}\left[R(0,[-N,N]^c)\right]\geq\mathds{E}\left[\sum_{i=-N}^{N-1}\xi_i\right] \geq c(\beta)N^{1-\beta},
$$
which then (together with a second moment computation) implies a polynomial lower bound on the resistance.
However, as $\beta$ increases, the number of long edges also increases, rendering the method of identifying cut-edges ineffective in providing a satisfactory lower bound.

In addition, as mentioned in Subsection \ref{rw}, \cite{M08} established a constant lower bound for the effective resistance associated with our $\beta$-LRP models. Indeed, the author showed the probability of the resistance being less than a certain constant is very low by choosing a specific test function in the dual variational formula of \eqref{eff-resis}.
However, obtaining a polynomial lower bound for the resistance through this method seems to be quite challenging, as a priori we have no  information about the form of the function that achieves the infimum in the dual of \eqref{eff-resis}.
A diverging lower bound was shown in \cite{B02} via constructing a collection of disjoint cut-sets, although the bound diverges rather slowly. In fact, the method of \cite{B02} can be seen as an application of multi-scale analysis, although the contribution obtained from each scale is barely large enough to obtain a diverging lower bound. The key contribution of our work, as we elaborate in what follows, is to employ a novel framework of multi-scale analysis.

The main idea of our approach for multi-scale analysis is to provide a lower bound for the effective resistance by combining the total energy generated by flows in \eqref{eff-resis} passing through ``good'' subintervals of \emph{different} scales within $[-N,N]$.
The novelty of our multi-scale analysis is largely captured by the application of the analysis fact that
\begin{equation}\label{main-eq}
a_{\alpha, n}\geq c\sum_{k=1}^{\lfloor \sqrt{n/\log n}\rfloor}\frac{a_{\alpha,n-k\log k}}{k\log k}\quad
\Longrightarrow\quad  a_{\alpha,n}\geq \widetilde{c}(\alpha)\e^{\delta(\alpha)n}
\end{equation}
for all $\alpha\in (0,1/2]$ and some constants $c>0$ and $\widetilde{c}(\alpha),\ \delta(\alpha)>0$ (depending only on $\beta$ and $\alpha$). Here  $n=\log_2 N$, and $a_{\alpha, k}$ is the $(1-\alpha)$-quantile of the ``effective resistance'' $\widehat{R}((-\infty,0],(2^k,+\infty))$ defined in \eqref{def-Rn} below.
The key challenge of this work is to prove the recursive formula on the left-hand side of \eqref{main-eq}, and one difficulty is that we have to simultaneously control all unit flows.
To this end, we partition the region near points where the flow enters the interval $(0, N]$ into intervals of length $2^i$ for  $i\geq 1$.
Then we search for ``cut-intervals'' at each step with $k\log k$ layers.
Here a cut-interval essentially plays the role of a cut-set, and roughly speaking it means that any unit flow from $(-\infty, 0]$ to $(N, \infty)$ must pass through the cut-interval (i.e., the flow must enter and then exit from the cut-interval). In our rigorous proof, the notion of cut-interval will be replaced by Definition \ref{def-g}.
In addition, we will show that a significant fraction of cut-intervals will be very good, in the sense that when the flow  passes through these intervals it generates a significant amount of energy (resulting in a significant contribution to effective resistances).
By using this and suitably selecting associated parameters, we can establish the left-hand side of \eqref{main-eq}, as incorporated in Proposition \ref{an-thm}.

Now, let us provide a slightly more detailed overview of the proof of \eqref{point} in Theorem \ref{mr}.
By a simple first moment computation, it is clear that with probability 1 there are only finitely many edges joining $[-N,N]$ and $[-N,N]^c$.
This implies that unit flows from 0 to $[-N,N]^c$  are well-defined, since they can be viewed as unit flows from $0$ to the finite set
$$
I=\left\{i\in [-N,N]^c:\ \exists j\in [-N,N]\ \text{such that }\langle i,j\rangle\in \mathcal{E}_{[-N,N]\times [-N,N]^c}\right\}.
$$
Now let $f$ be a unit flow from 0 to $[-N,N]^c$ attaining the infimum in \eqref{eff-resis}.
Then, we can see that the flow $f$ emanates from 0 (with an outflow of $f_{(0)}=\sum_{i:0\sim i}f_{0i}=1$) and ultimately flows into $[-N,N]^c$ (with an inflow of $\sum_{j\in [-N,N]^c}\sum_{i\in [-N,N]: i\sim j}f_{ij}=1$).
As a result, the flow $f$ must pass through intervals $[-N,0)$ or $(0,N]$ (unless 0 is directly connected to $[-N, N]^c$, which is unlikely).
So let us define $\theta_1$ as the portion of flow $f$ that passes through the interval $(0,N]$ and then flow into $[-N,N]^c$. Similarly, we define $\theta_2$ by replacing $(0,N]$ with $[-N,0)$ in the definition of $\theta_1$. We also let
\begin{equation}\label{flow+-}
|\theta_1|=\sum_{i\in (0,N]}\sum_{j\in [-N,N]^c}f_{ij}\quad \text{and}
\quad |\theta_2|=\sum_{i\in [-N,0)}\sum_{j\in [-N,N]^c}f_{ij}
\end{equation}
represent amounts of flows exiting from intervals $(0,N]$ and $[-N,0)$, respectively. It is clear that $\max\{|\theta_1|,|\theta_2|\}\geq 1/2$. Define
\begin{equation}\label{def-Rtilde}
\begin{aligned}
\widehat{R}([-N,0],[-N,N]^c)
&=\inf\left\{\frac{1}{2}\sum_{u\sim v }f^2_{uv}:\ f\ \text{is a unit flow from $[-N,0]$ to }[-N,N]^c,\right.\\
&\quad \quad \quad\quad \quad \quad \quad \quad \quad  \text{and }f_{uv}=0\ \text{for all }\langle u,v\rangle\in \mathcal{E}_{[-N,0]\times [-N,N]^c} \Bigg\}
\end{aligned}
\end{equation}
as the effective resistance generated by unit flows from $[-N,0]$ to $[-N,N]^c$, passing through the interval $(0,N]$. Similarly, denote $\widehat{R}([0,N],[-N,N]^c)$ as the effective resistance generated by unit flows from $[0,N]$ to $[-N,N]^c$, passing through the interval $[-N,0)$.
Then we can deduce that
\begin{equation}\label{Rn-Rhat}
R(0,[-N,N]^c)\geq \max\left\{|\theta_1|^2 \widehat{R}([-N,0],[-N,N]^c), |\theta_2|^2\widehat{R}([0,N],[-N,N]^c)\right\}.
\end{equation}
\begin{figure}[htbp]
\centering
\includegraphics[height=1.5cm,width=14cm]{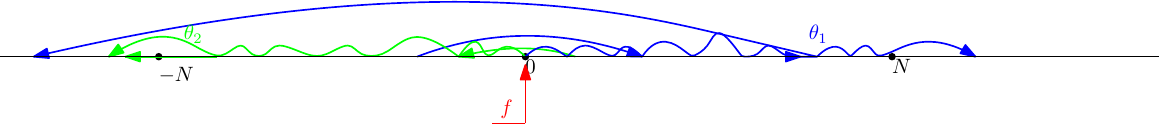}
\caption{The illustration for \eqref{Rn-Rhat}. The red arrow represents the unit flow $f$ entering at 0. The blue curves represent the flows exiting from $(0,N]$, while the green curves represent the flows exiting from $[-N,0)$.}
\label{flow1}
\end{figure}
(See Figure \ref{flow1} for an illustration). Note that $\widehat{R}([-N,0],[-N,N]^c)$ and $\widehat{R}([0,N],[-N,N]^c)$ have the same distribution. Consequently, it suffices to prove that $\widehat{R}([-N,0],[-N,N]^c)$ has a polynomial lower bound with high probability.

Furthermore, for the sake of convenience in employing an iterative approach for estimating resistances, we lower-bound  $\widehat{R}([-N,0],[-N,N]^c)$ via the resistance generated by unit flows  from $(-\infty,0]$ to $(N,+\infty)$ passing through the interval $(0,N]$, which is defined as
\begin{equation}\label{def-Rn}
\begin{aligned}
\widehat{R}((-\infty,0],(N,+\infty))
&=\inf\left\{\frac{1}{2}\sum_{u\sim v }f^2_{uv}:\ f\ \text{is a unit flow from $(-\infty,0]$ to }(N,+\infty)\right.\\
&\quad \quad \quad\quad \quad \quad \quad \quad \quad  \text{and }f_{uv}=0\ \text{for all }\langle u,v\rangle\in \mathcal{E}_{(-\infty,0]\times (N,+\infty)} \Bigg\}.
\end{aligned}
\end{equation}
(See Subsection \ref{propertyRhat} for more details). In conclusion, our primary focus is to establish a polynomial lower bound for the resistance $\widehat{R}((-\infty,0],(N,+\infty))$ as incorporated in Theorem \ref{Rn}.

To prove Theorem \ref{Rn}, we first introduce the concept of good pairs of intervals.
Roughly speaking, for any $z\in \mathds{Z}$ and $i\in \mathds{N}$, we say the pair of intervals $([z-2^{i+1},z-2^i),(z+2^i,z+2^{i+1}])$ is good if any path originating from $z$ to $[z-2^{i+1},z+2^{i+1}]^c$
must pass through either $[z-2^{i+1},z-2^i)$ or $(z+2^i,z+2^{i+1}]$ (see Definition \ref{def-g}).
Essentially, we may consider good pairs of intervals as some kind of ``cut-intervals''.
Hence, when the flow exits from $z$, it must pass through at least one of these two intervals, generating a certain amount of energy.
We will employ the Firework process theory (see e.g.\ \cite{GGJR14}) by viewing coverage of long edges over intervals as the propagation in the Firework process, to obtain some estimates for the distribution of the number of intervals that can be covered at a time (see Lemma \ref{Rk}).
The definition of an interval covered by long edges will be provided in \eqref{cover} below.
From this and a general Chernoff-Hoeffding bound, we can show that
for a fixed $z$, with high probability there exist a sufficient number of good pairs of intervals at various scales around it (see Proposition \ref{ld-xi}).

We next consider the area surrounding points where the unit flow (from $(-\infty,0]$ to $(N,+\infty))$ enters $(0,N]$. We segment such area into intervals of length $2^i$ for $i\geq 1$, which we will refer to as the $i$-th layer intervals.
Then we search for good pairs of intervals at each step with $k\log k$ layers and show that with high probability we can find enough good pairs of intervals in many layers (see Figure \ref{explore}).
Note that the $k\log k$ here is in correspondence with $n-k \log k$ in the subscript in \eqref{main-eq} and is also responsible for the factor of $k \log k$ in the denominator there.
Therefore, when the flow passes through those good intervals, it will provide sufficiently large energies (i.e.\ effective resistances).
This implies an effective lower bound for the resistance (see Lemma \ref{induct-thm}).
Then we can establish a recursive formula for quantiles of the resistance $\widehat{R}(\cdot,\cdot)$ (see Proposition \ref{an-thm}).
From this we can  obtain a polynomial lower bound in $N$ for the resistance $\widehat{R}((-\infty,0],(N,+\infty))$ as in  Theorem \ref{Rn}.
We re-iterate that one difficulty is that we have to simultaneously control all unit flows. To achieve this, we show that good pairs of intervals found above through multi-scale analysis are essentially cut-sets, so all unit flows must pass through these intervals and then generate enough energy. This is incorporated in Section \ref{sect-Rhat}.

\bigskip


{\bf Notational conventions.}
We denote $\mathds{N}=\{1,2,3.\cdots \}$. For $a<b$, we define $[a,b]_{\mathds{Z}}=[a,b]\cap \mathds{Z}$. When we refer to an interval $I$, it always implies that $I\cap \mathds{Z}$. For any two disjoint sets $I_1,I_2\subset \mathds{Z}$, let $f$ be a flow from $I_1$ to $I_2$, we write $|f|$ for its amount, that is,
\begin{equation}\label{size}
|f|=\sum_{i\in I_1}\sum_{j\in I_1^c}f_{ij}.
\end{equation}
For example, if $f$ is a unit flow, then $|f|=1$. In addition, for any two sets $I_1,I_2\subset \mathds{Z}$, we recall $\mathcal{E}_{I_1\times I_2}:=\mathcal{E}\cap(I_1\times I_2)$.

Throughout the paper, we use $c_1, c_2,c,c', \cdots$ to denote positive constants, whose values are the same within each section but may vary from section to section.

\section{Good intervals and associated estimates}\label{sect-goodi}

 For $z\in \mathds{Z}$ and $i\in \mathds{N}$, denote by $I_i^+(z)$ and $I_i^-(z)$ the intervals $(z+2^i,z+2^{i+1}]$ and $[z-2^{i+1},z-2^i)$, respectively.
\begin{definition}\label{def-g}
We say the pair of intervals $(I_i^-(z),I_i^+(z))$ is \textit{good} if
\begin{itemize}
\item[(1)]$[z-2^{i+1},z+2^i]$ and $(z+2^{i+1},+\infty)$ are not directly connected by any long edge;
\medskip

\item[(2)] $(-\infty,z-2^{i+1})$ and $[z-2^i, z+2^{i+1}]$ are not directly connected by any long edge.
\end{itemize}
\end{definition}

\begin{figure}[htbp]
\centering
\includegraphics[height=1.5cm,width=16cm]{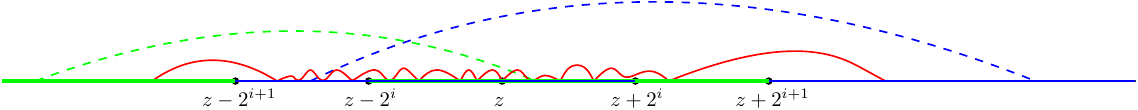}
\caption{The illustration for the definition of the good pair of intervals. The blue dashed curve in the graph represents the absence of long edges directly connecting intervals $[z-2^{i+1},z+2^i]$ and $[z+2^{i+1},+\infty)$ (blue lines), while the green dashed curve represents the absence of long edges directly connecting intervals $(-\infty,z-2^{i+1}]$ and $[z-2^i, z+2^{i+1}]$ (green lines). Therefore, for any path starting from $z$ to $[z-2^{i+1},z+2^{i+1}]^c$ (the red curve), it must pass through either $I_i^-(z)$ or $I_i^+(z)$.}
\label{goodi}
\end{figure}

In the following, we aim to provide a large deviation estimate for the number of good pairs of intervals. To do this, we only need to consider the case where $z=0$ due to the translation invariance of our model. For simplicity, we denote $I^\pm_i(0)$ as $I^\pm_i$ in this case. We now define a sequence of Bernoulli random variables $\{\xi_i\}_{i\geq1}$ as
\begin{equation*}
\xi_i=\begin{cases}
1,\quad & (I_i^-,I_i^+)\ \text{is good,}\\
0,\quad & \text{otherwise}.
\end{cases}
\end{equation*}

\begin{proposition}\label{ld-xi}
For any $\beta>0$, there exists a constant $\kappa=\kappa(\beta)\in (0,1)$ {\rm(}depending only on $\beta${\rm)} such that for all $n\geq 1$, we have
$$
\mathds{P}\left[\sum_{i=1}^{n}\xi_i\leq (1-\kappa)n/2\right]\leq \exp\left\{-(1-\kappa)^2n/2\right\}.
$$
\end{proposition}

To prove Proposition \ref{ld-xi}, we need some preparations.

\begin{lemma}\label{induct-w-lem}
There is a constant  $\kappa \in (0,1)$ {\rm(}depending only on $\beta${\rm)} such that the following holds. For any $r\in \mathds{N}$ and any  $S=\{i_1,i_2,\cdots,i_r\}\subset \mathds{N}$ in the ascending order, let
$$
w_{r}=\mathds{P}[\xi_{i_1}=0,\xi_{i_2}=0,\cdots,\xi_{i_r}=0].
$$
Then $w_r\leq \kappa^r$.
\end{lemma}

To prove Lemma \ref{induct-w-lem}, we will review the Firework process (see e.g.\ \cite{GGJR14}) and view coverage of  long edges over intervals as the propagation in the Firework process. For that, we fix $r\in \mathds{N}$ and $S=\{i_1,i_2,\cdots,i_r\}\subset \mathds{N}$ being a set in the ascending order.

For $k\in[2, r+1]_{\mathds{Z}}$,  let $L_k$ represent the minimum number of $s\leq k-1$ such that there exists at least one long edge within
$$
\mathcal{E}_{[-2^{i_k+1}, -2^{i_{k-1}+1})\times [-2^{i_s},2^{i_s+1}]}\cup\mathcal{E}_{(2^{i_{k-1}+1},2^{i_k+1}]\times [-2^{i_s+1},2^{i_s}]},
$$
where $i_{r+1}:=+\infty$ (see Figure \ref{Rk-figure} for an illustration). It is worth emphasizing that $L_1,\cdots, L_{k+1}$ are independent from the independence of edges in our model. We say the pair of intervals $(I_{i_s}^-,I_{i_s}^+)$ is covered by long edges if there exists at least one long edge within
\begin{equation}\label{cover}
\mathcal{E}_{(-\infty,-2^{i_s+1})\times [-2^{i_s},2^{i_s+1}]}\cup \mathcal{E}_{(2^{i_s+1},+\infty)\times [-2^{i_s+1},2^{i_s}]}.
\end{equation}

\begin{figure}[htbp]
\centering
\includegraphics[height=3.8cm,width=16cm]{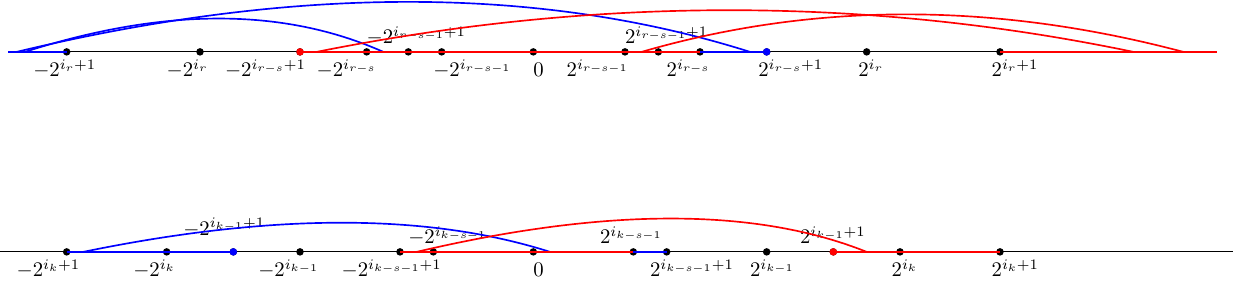}
\caption{The illustrations for the definitions of $L_{r+1}$ (at the top) and $L_k,\ k\in [2,r]_{\mathds{Z}}$ (at the bottom).
For $k\in [2,r+1]_{\mathds{Z}}$ and $s\in [0,k-1]_{\mathds{Z}}$, the event $L_k>s$ is equivalent to the event that there exists at least one long edge in $\mathcal{E}_{[-2^{i_k+1},-2^{i_{k-1}+1})\times [-2^{i_{k-s-1}},2^{i_{k-s-1}+1}]}$ (blue curves) or $\mathcal{E}_{(2^{i_{k-1}+1},2^{i_k+1}]\times [-2^{i_{k-s-1}+1},2^{i_{k-s-1}}]}$ (red curves).
}
\label{Rk-figure}
\end{figure}

For the distribution of $L_k$, we have the following property.

\begin{lemma}\label{Rk}
For $r\in \mathds{N}$ and $S=\{i_1,\cdots, i_r\}\subset \mathds{N}$ in the ascending order, we have
$$
\mathds{P}[L_k> s]\leq 1-\left(1-\frac{3}{2^{s+1}+2}\right)^{2\beta}
$$
for all $k\in [2,r+1]_{\mathds{Z}}$ and $s\in [0,k-1]_{\mathds{Z}}$.
\end{lemma}

\begin{proof}
 It follows from the definition of $L_k$ that for each $k\in [2,r+1]_{\mathds{Z}}$ and each $s\in [0,k-1]_{\mathds{Z}}$,
\begin{equation*}
\begin{aligned}
\mathds{P}[L_k>s]&=1-\exp\left\{-2\beta\int_{-2^{i_k+1}}^{-2^{i_{k-1}+1}}\int_{-2^{i_{k-s-1}}}^{2^{i_{k-s-1}+1}}\frac{1}{(u-v)^2}\d u\d v\right\}\\
&=1-\exp\left\{-2\beta\int_{-2^{i_k-i_{k-s-1}+1}}^{-2^{i_{k-1}-i_{k-s-1}+1}}\int_{-1}^{2}\frac{1}{(u-v)^2}\d u\d v\right\}\\
&\leq 1-\left(1-\frac{3}{2^{i_{k-1}-i_{k-s-1}+1}+2}\right)^{2\beta}\leq 1-\left(1-\frac{3}{2^{s+1}+2}\right)^{2\beta},
\end{aligned}
\end{equation*}
where the last inequality is from $i_{k-1}-i_{k-s-1}\ge s$.
\end{proof}

We let $A_0=\{i_r+1\}$ and for $m\geq 1$, we inductively define
\begin{equation}\label{An}
A_m=\left\{i_s\in S:\ \text{there exists }i_k\in A_{m-1}\ \text{such that }s\in \{k-L_k,\cdots,k-1\}\right\}\setminus A_{m-1}.
\end{equation}
The above definition in fact corresponds to a ``spreading'' procedure of the edge set for the LRP model, where in the $m$-th step we explore long edges in
\begin{equation}\label{le-Am}
\cup_{i_k\in A_{m-1}}\left(\mathcal{E}_{(-2^{i_k+1}, -2^{i_{k-1}+1}]\times [-2^{i_s},2^{i_s+1}]}\cup\mathcal{E}_{[2^{i_{k-1}+1},2^{i_k+1})\times [-2^{i_s+1},2^{i_s}]}\right)
\end{equation}
to determine if the element $i_s$ is in $A_m$. That is, if the edge set in \eqref{le-Am} is non-empty, then  $i_s\in A_m$. We see that $i_s\in A_m$ (namely, the pair of intervals $(I_{i_s}^-,I_{i_s}^-)$ is newly covered at the $m$-th step) if it was not covered by the edge set
\begin{equation*}
\cup_{i_k\in A_{m-2}}\left(\mathcal{E}_{(-2^{i_k+1}, -2^{i_{k-1}+1}]\times [-2^{i_s},2^{i_s+1}]}\cup\mathcal{E}_{[2^{i_{k-1}+1},2^{i_k+1})\times [-2^{i_s+1},2^{i_s}]}\right),
\end{equation*}
but covered by the edge set in \eqref{le-Am}. The ``spreading procedure'' will stop upon $A_m=\varnothing$ and from \eqref{An} we can see that $A_{m'}=\varnothing$ for all $m'>m$ if $A_m=\varnothing$. Moreover, $\cup_{m\geq 0}A_m\subset S\cup \{i_r+1\}$ represents the set of subscripts for intervals (``spreaders'') at the end of the above spreading procedure. Let
\begin{equation}\label{Mr}
M_r:=\min\left\{k\in \mathds{N}:\ i_k\in \cup_{m\geq 0}A_m\right\}
\end{equation}
be the subscript of the last pair of intervals that are covered in this spreading procedure. Combining this with definitions of good pair of intervals and $\xi_k$, we can see that
\begin{equation}\label{Mr-zeta}
\mathds{P}[\xi_{i_1}=0,\cdots,\xi_{i_r}=0]\leq \mathds{P}[M_r\geq r].
\end{equation}
(See Figure \ref{Mr-figure} for an illustration).
We will use some estimates for the Firework process (see e.g.\ \cite{GGJR14}) to bound the right-hand side of \eqref{Mr-zeta} from above, which in turn provides an upper bound on the left-hand side of \eqref{Mr-zeta}.

\begin{figure}[htbp]
\centering
\includegraphics[height=1cm,width=15cm]{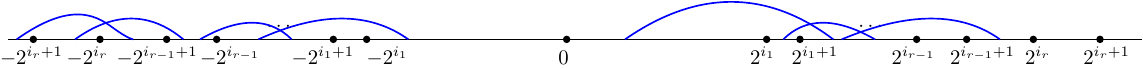}
\caption{The illustration for \eqref{Mr-zeta}. The event $\{\xi_{i_1}=0,\cdots,\xi_{i_r}=0\}$ is equivalent to that $(I_{i_1}^+,I_{i_1}^-),\cdots,(I_{i_r}^+,I_{i_r}^-)$ are all covered by long edges. Thus we have  $\{\xi_{i_1}=0,\cdots,\xi_{i_r}=0\}\subset\{M_r\geq r\}$.}
\label{Mr-figure}
\end{figure}

\begin{lemma}\label{Mr-upper}
There exists $\kappa \in (0,1)$ {\rm(}depending only on $\beta${\rm)} such that for all $r\geq 1$,
$$
\mathds{P}[M_r\geq r]\leq \kappa^r.
$$
\end{lemma}

\begin{proof}
For $k\in[2,r+1]_{\mathds{Z}}$ and $s\in [0,k-1]_{\mathds{Z}}$, denote
$$
\alpha_k(s)=\mathds{P}[L_k\leq s]=1-\mathds{P}[L_k>s].
$$
Then from Lemma \ref{Rk}, we can see that for each $s\in [0,r]_{\mathds{Z}}$,
\begin{equation}\label{alpha-r}
\min_{s<k\leq r+1}\alpha_k(s)\geq \left(1-\frac{3}{2^{s+1}+2}\right)^{2\beta}=:\widetilde{\alpha}(s).
\end{equation}
Moreover, from the fact that $\log(1-x)\geq -2x$ for all $x\in (0,3/4)$, we can see that
\begin{equation}\label{alpha-tilde}
\begin{aligned}
\log\prod_{s\geq 0}\widetilde{\alpha}(s)&=2\beta\sum_{s\geq 0}\log\left(1-\frac{3}{2^{s+1}+2}\right)\geq -12\beta \sum_{s\geq 0}\frac{1}{2^{s+1}}=-12\beta.
\end{aligned}
\end{equation}

We now consider a sequence of i.i.d.\ random variables $\widetilde{L}_{1},\cdots, \widetilde{L}_{r+1}$ with the distribution $\widetilde{\alpha}(s)$, and define $\widetilde{A}_m$ and $\widetilde{M_r}$ according to \eqref{An} and \eqref{Mr} by replacing $L_k$ with $\widetilde{L}_k$, respectively.
Then from \eqref{alpha-r} and the independence of $L_1,\cdots, L_{r+1}$ as previously mentioned before \eqref{cover}, we can see that
\begin{equation}\label{Mn-tilde}
\mathds{P}[M_r\geq s]\leq \mathds{P}\left[\widetilde{M}_r\geq s\right]\quad  \text{for all } s\in [0,r]_{\mathds{Z}}.
\end{equation}
Additionally, it follows from \eqref{alpha-r} and \eqref{alpha-tilde} that $\widetilde{\alpha}(s)$ increases exponentially to 1 as $s\to \infty$ and $\prod_{s\geq 0}\widetilde{\alpha}(s)\geq \e^{-12\beta}>0$, which implies that conditions stated in \cite[Proposition 1]{GGJR14} are satisfied. Consequently, by applying \cite[Proposition 1]{GGJR14} to $\widetilde{\alpha}(s)$, we get that there exist $\kappa_1\in (0,1)$ and $R>0$ (both depending only on $\beta$) such that for all $r\geq R$,
\begin{equation}\label{kappa1}
 \mathds{P}\left[\widetilde{M}_r\geq r\right]\leq \kappa_1^r.
\end{equation}
Moreover, choose $\kappa_2\in (0,1)$ (depending only on $\beta$) such that $\kappa_2^R\geq 1-4^{-\beta}$. By combining this with Lemma \ref{Rk}, we see that for all $r<R$,
\begin{equation}\label{kappa2}
\mathds{P}\left[\widetilde{M}_r\geq r\right]\leq 1-\mathds{P}[L_{r+1}=0]\leq 1-4^{-\beta}\leq \kappa_2^r.
\end{equation}
Let us denote $\kappa=\max\{\kappa_1,\kappa_2\}$. Then \eqref{kappa1} and \eqref{kappa2} yield that $\mathds{P}\left[\widetilde{M}_r\geq r\right]\leq \kappa^r$ for all $r\geq 1$. Hence, we obtain the desired statement by combining this with \eqref{Mn-tilde}.
\end{proof}

With the above lemmas at hand, we can provide the
\begin{proof}[Proof of Lemma \ref{induct-w-lem}]
Recall that $M_r$ is defined in \eqref{Mr}. Then from it and Lemma \ref{Mr-upper}, we arrive at
$$
w_{r}=\mathds{P}[\xi_{i_1}=0,\xi_{i_2}=0,\cdots,\xi_{i_r}=0]\leq  \mathds{P}\left[M_r\geq r\right]\leq \kappa^r,
$$
where $\kappa$  is the constant in Lemma \ref{Mr-upper}, depending only on $\beta$.
\end{proof}

Combining the above lemmas with \cite[Theorem 3.4]{AA97} (see also \cite[Theorem 1]{RV10}), we can present the

\begin{proof}[Proof of Proposition \ref{ld-xi}]
Lemma \ref{induct-w-lem} implies that conditions stated in \cite[Theorem 1]{RV10} are satisfied with $\delta=\kappa\in (0,1)$. Thus, by applying \cite[Theorem 1]{RV10} with $\delta=\kappa$ and $\gamma =(1+\kappa)/2$, we conclude that
\begin{equation}\label{xi-lt}
\mathds{P}\left[\sum_{i=1}^{n-1}\xi_i\leq (1-\kappa)n/2\right]=\mathds{P}\left[\sum_{i=1}^{n-1}(1-\xi_i)\geq \gamma n\right]\leq \exp\left\{-nD(\gamma \| \delta)\right\},
\end{equation}
where $D(\gamma \| \delta):=\gamma\log\frac{\gamma}{\delta}+(1-\gamma)\log\frac{1-\gamma}{1-\delta}$ is the binary relative entropy between $\gamma$ and $\delta$.
In addition, note that $D(\gamma\|\delta)\geq 2(\gamma-\delta)^2=(1-\kappa)^2/2$. Plugging this into \eqref{xi-lt} yields the desired result.
\end{proof}

\section{The estimates of effective resistance $\widehat{R}_n$}\label{sect-Rhat}

For $n\geq1$, let $N=2^n$ and let $\widehat{R}_n$ be the effective resistance generated by unit flows from $(-\infty,0]$ to $(N,+\infty)$, passing through the interval $(0,N]$. That is,
\begin{equation}\label{def-Rn-2}
\begin{aligned}
\widehat{R}_n&=\inf\left\{\frac{1}{2}\sum_{u\sim v}f^2_{uv}:\ f\ \text{is a unit flow from $(-\infty,0]$ to }(N,+\infty)\right.\\
&\quad \quad \quad\quad \quad \quad \quad \quad \quad  \text{and }f_{uv}=0\ \text{for all }\langle u,v\rangle\in \mathcal{E}_{(-\infty,0]\times (N,+\infty)} \Bigg\}.
\end{aligned}
\end{equation}
It is obvious that $\widehat{R}_n$ depends only on edges where at least one of their endpoints falls within the interval $(0,N]$.

The main focus of this section is to establish an exponential lower bound for $\widehat{R}_n$ (note that exponential in $n$ is consistent with polynomial in the size of the interval).

\begin{theorem}\label{Rn}
For any $\beta>0$, there is a constant $\delta>0$ {\rm(}depending only on $\beta${\rm)} such that the following holds.
For any $\varepsilon\in(0,1/2]$, there exists a constant $c=c(\varepsilon)>0$ {\rm(}depending only on $\beta$ and $\varepsilon${\rm)} such that for all $n\geq 1$,
$$
\mathds{P}\left[\widehat{R}_n\geq c{\rm e}^{\delta n}\right]\geq 1-\varepsilon.
$$
\end{theorem}

The proof of Theorem \ref{Rn} will be provided in Subsection \ref{proof-mr1}. A main input for proving Theorem \ref{Rn} is to lower-bound the quantile of $\widehat{R}_n$, as incorporated in Proposition \ref{quantile-lb} below. To be precise, for $n>1$ and $\alpha \in (0, 1)$, we define the $(1-\alpha)$-quantile of $\widehat{R}_n$ as
\begin{equation}\label{def-epsi}
a_{\alpha,n}=\sup\left\{r\geq 0: \mathds{P}\left[\widehat{R}_n\geq r\right]\geq 1-\alpha\right\}.
\end{equation}

\begin{proposition}\label{quantile-lb}
For any $\alpha\in(0,1/2]$, there exist constants $c=c(\alpha)>0$ and $\delta=\delta(\alpha)>0$ {\rm(}both depending only on $\beta$ and $\alpha${\rm)} such that
$a_{\alpha,n}\geq c{\rm e}^{\delta n}$ for all $n\geq 1$.
\end{proposition}

The proof of Proposition \ref{quantile-lb} will also be included in Subsection \ref{proof-mr1}.
Note that the estimate in Proposition \ref{quantile-lb} is a weaker version of Theorem \ref{Rn}. This distinction arises from the fact that the parameter $\delta$ in Proposition \ref{quantile-lb} is allowed to depend on the parameter $\alpha$, whereas $\delta$ in Theorem \ref{Rn} only depends on $\beta$.

Hereon, we provide a general overview on our proof of Theorem \ref{Rn}, which encapsulates the primary challenge.
Initially, for a fixed unit flow from $(-\infty,0]$ to $(N,+\infty)$ passing through $(0,N]$,
we label points where the unit flow enters $(0,N]$ as $z_1, z_2,...,z_{\overline{\eta}}$.
Here $\overline{\eta}$ represents a random variable taking values on $\mathds{N}$ (see specific definitions in the final part of this subsection).
The area surrounding these points is segmented into intervals of length $2^i$ for $i\geq 1$, which we will refer to as the $i$-th layer intervals.

First, we say an interval is very good if any unit flow must traverse it and generate substantial energy within it (see Definition \ref{def-vg}).
In search of very good intervals, we will investigate intervals in different layers step by step. The exact number of layers in each step is determined by a sequence $\{b_k\}$ (see \eqref{def-bk}), whose definition includes two parameters $L$ and $M$ to facilitate our adjustments later.
Using the estimate for the number of good intervals in Proposition \ref{ld-xi}, it can be inferred that with high probability we can find at least one very good pair of intervals near every inflow point in each step, as incorporated in Proposition \ref{Prob-vg}. This event will be defined as  $E_{\alpha,n}$ in Definition \ref{def-E} below.

Next, a recursive formula for the $(1-\alpha)$-quantiles $\{a_{\alpha,i}\}_{ i\geq 1}$ of the effective resistances $\{\widehat{R}_i\}_{i\geq 1}$ will be established in Subsection \ref{Rn-sect}.
To arrive at this, we will show that on the event $E_{\alpha,n}\cap F_{\alpha,n}$ (here $F_{\alpha,n}$ is an event to ensure that $z_{\overline{\eta}}$ is not close to the point $2^n$, see Lemma \ref{z-eta}), any flow must pass through many very good intervals,
each of which provides sufficiently large energy and thus result in a significant contribution to the effective resistance.
This will allow us to obtain an effective lower bound for the resistance $\widehat{R}_n$ (see  Lemma \ref{induct-thm}). Thus by appropriately selecting parameters $M$ and $L$, we can establish the recursive formula for the $(1-\alpha)$-quantiles $\{a_{\alpha,i}\}_{i\geq 1}$ in Proposition \ref{an-thm}.

After that, we will establish an exponential lower bound for the $(1-\alpha)$-quantile $a_{\alpha,n}$ from the recursive formula in Proposition \ref{an-thm}. Combining this with appropriately selecting parameters $M$ and $L$, we will complete the proof of Theorem \ref{Rn}.  This is included in Subsection \ref{proof-mr1}.

In the final of this part, we will introduce some notations which will be used repeatedly.
Fix a sufficiently large $n\in \mathds{N}$ and recall that $N=2^n$.
For convenience, let $\mathcal{E}_n=\mathcal{E}_{(-\infty,0]\times (0,N]}$ denote the set of all  edges  with one endpoint in the interval $(-\infty,0]$ and the other endpoint in $(0,N]$, and denote $Z$ as the collection of those endpoints in $(0,N]$.  Formally,
\begin{equation}\label{def-Z}
Z=\left\{z\in (0,N]_{\mathds{Z}}:\ \exists u\in (-\infty,0]_{\mathds{Z}}\ \text{such that }\langle u,z\rangle\in \mathcal{E} \right\}.
\end{equation}
Clearly, $1\in Z$ since $\langle 0,1\rangle\in \mathcal{E}$.
Additionally, for $i\in[1,n]_{\mathds{Z}}$, let $\eta_i$ denote the number of points in $Z\cap(2^{i-1},2^{i}]$. More specifically, for $u,v\in \mathds{Z}$, denote
\begin{equation}\label{xi-uv}
\xi_{uv}=
\begin{cases}
1,\quad &\langle u,v\rangle\in \mathcal{E},\\
0,\quad &\text{otherwise}.
\end{cases}
\end{equation}
Then $\eta_i$ can be expressed as
\begin{equation*}\label{eta-i}
\eta_i=\sum_{2^{i-1}<v\leq 2^i}\left[1-\prod_{u\leq 0}\left(1- \xi_{uv}\right)\right].
\end{equation*}
According to the independence of edges, it follows that $\eta_1, ..., \eta_{n}$ are independent random variables with expectations
\begin{equation*}\label{mu-i}
\begin{aligned}
\mu_i:=\mathds{E}[\eta_i]&=\sum_{2^{i-1}<v\leq 2^i}\left[1-\prod_{u\leq 0}\mathds{E}\left(1- \xi_{uv}\right)\right]=\sum_{2^{i-1}<v\leq 2^i}\left[1-\exp\left\{-\beta\int_{-\infty}^{1}\int_v^{v+1}\frac{1}{|x-y|^2}\d x\d y\right\}\right]
\end{aligned}
\end{equation*}
for $i\in [1,n]_{\mathds{Z}}$. By a simple calculation, we can see that there exists a constant $c=c(\beta)\in (0,1)$ (depending only on $\beta$) such that
 \begin{equation}\label{mu-i}
 c\beta\log 2\leq  \mu_i\leq \beta\log 2\quad \text{for all } i\in [1,n]_{\mathds{Z}}.
 \end{equation}
We also denote $\overline{\eta}=\sum_{l=1}^{n}\eta_l$. It is clear that
\begin{equation*}\label{mu}
c\beta n\log 2\leq \mathds{E}[\overline{\eta}]\leq \beta n\log 2.
\end{equation*}
We now sort points in $Z$ in the ascending order and denote them as $z_0=1,z_1,z_2\cdots, z_{\overline{\eta}}$.

Denote
\begin{equation}\label{I+-}
I^-_{i,j}=[z_i-2^{j+1},z_i-2^{j}) \quad \text{and}\quad I^+_{i,j}=(z_i+2^{j},z_i+2^{j+1}]\quad \text {for all } i\in [0,\overline{\eta}],\ j\geq 0.
\end{equation}
Write $\{b_k\}_{k\geq 0}$ for the sequence satisfying
\begin{equation}\label{def-bk}
b_0=n,\quad b_k=b_{k-1}-M -L\log (n+M-b_{k-1})\quad \text{for } k\geq 1,
\end{equation}
where $M$ and $L$ 
are two constants which will be determined later (see \eqref{cond-L} and \eqref{E-M} below). Note that $\{b_k\}$ is dependent on $n$, but for the sake of brevity, the notation we are using here does not reflect this.
We also define
\begin{equation}\label{def-etabar}
\overline{\eta}_{b_k}=1+\sum_{l=1}^{b_k}\eta_l
\end{equation}
and
\begin{equation}\label{def-Kn}
K_n=\sup\left\{k\geq 0:\ b_k\geq 0\right\}.
\end{equation}

\subsection{Very good intervals and associated estimates}\label{En}

We begin by extending the definition of $\widehat{R}_n$ to a more general definition of effective resistance produced by unit flows passing through an interval. Specifically, for two intervals $I_1=[x_1,x_2]$ and $I_2=[x_3,x_4]$ with $-\infty\leq x_1<x_2<x_3<x_4\leq +\infty$, define
\begin{equation}\label{def-Rhat}
\begin{aligned}
\widehat{R}(I_1,I_2)&=\inf\left\{\frac{1}{2}\sum_{u\sim v}f_{uv}^2:\ f\ \text{is a unit flow from $I_1$ to }I_2\ \text{and }\right.\\
&\quad \quad \quad\quad   f_{uv}=0\ \text{for all }\langle u,v\rangle\in \mathcal{E}_{(-\infty,x_2]\times [x_3,+\infty)}\cup\mathcal{E}_{[x_1,x_4]^c\times (x_2,x_3)} \Bigg\}.
\end{aligned}
\end{equation}
We will refer to $\widehat{R}(I_1,I_2)$ as the effective resistance generated by unit flows (confined to $[x_1,x_4]$) passing through the interval $[x_2,x_3]$.
Moreover, from \eqref{def-Rn-2} and \eqref{def-Rhat}, it is obvious that $\widehat{R}((-\infty,0],(N,+\infty))=\widehat{R}_n$. The definitions of $\widehat{R}(\cdot,\cdot)$ and $R(\cdot,\cdot)$ imply that
$$\widehat{R}_n\geq R((-\infty,0],(N,+\infty)).$$
In particular, when $\mathcal{E}_{(-\infty,0]\times (N,+\infty)}=\varnothing$, we have $\widehat{R}_n= R((-\infty,0],(N,+\infty))$.

It is worth emphasizing that if we remove the edge set $\mathcal{E}_{(-\infty,x_2]\times [x_3,+\infty)}\cup\mathcal{E}_{[x_1,x_4]^c\times (x_2,x_3)}$ from the $\beta$-LRP model and view it as a new graph, then $\widehat{R}(I_1,I_2)$ becomes the classical effective resistance on this new graph. Thus,
$\widehat{R}(\cdot,\cdot)$ clearly possesses the fundamental properties of the effective resistance, as presented in the following lemma.

\begin{lemma}\label{mono}
For any two intervals $I_1,I_2\subset \mathds{Z}$ with $I_1\cap I_2=\varnothing$, we have
\begin{itemize}
\item[\rm (1)] $\widehat{R}(I_1,I_2)\geq R(I_1,I_2)$;
\medskip

\item[\rm(2)] $\widehat{R}(I_1,I_3)\geq \widehat{R}(I_1,I_2)$ for all intervals $I_3\subset I_2$.

\end{itemize}
\end{lemma}

\begin{proof}
The assertion in (1) can be deduced by combining \eqref{def-Rhat} and \eqref{eff-resis},
while the assertion in (2) can be derived directly from \eqref{def-Rhat}.
\end{proof}

Recall that $Z=\{z_0,z_1,\cdots, z_{\overline{\eta}}\}$  is defined in \eqref{def-Z}.
We now introduce the definition of ``very good'' intervals as follows.
\begin{definition}\label{def-vg}
Fix $\alpha\in (0,1/2]$. For $ i\in [0,\overline{\eta}]_{\mathds{Z}}$ and $j\geq 0$, we say the pair of intervals $(I_{i,j}^-,I_{i,j}^+)$ is \textit{$\alpha$-very good} if
\begin{itemize}


\item[(1)] $[(z_i-2^{j+1})\vee 0,\ z_i+2^{j}]$ and $(z_i+2^{j+1},+\infty)$ are not directly connected by any long edge, as well as $(0,\ (z_i-2^{j+1})\vee 0)$ and $[z_i-2^{j},z_i+2^{j+1}]$ are not directly connected by any  long edge;
\medskip

\item[(2)] the resistance $\widehat{R}([z_i-2^{j+1},z_i+2^{j}],(z_i+2^{j+1},+\infty))\geq a_{\alpha,j}$;

\medskip

\item[(3)] the resistance $\widehat{R}((z_i-2^{j},z_i+2^{j+1}],(-\infty,z_i-2^{j+1}])\geq a_{\alpha,j}$.


\end{itemize}

\end{definition}

It is worth emphasizing that the event in (1) is a modified version of the definition of  good pair of intervals in Definition \ref{def-g}.
Indeed, it can be observed that the event in Definition \ref{def-vg} (1) contains the event $\{(I_{i,j}^-,I_{i,j}^+)\text{ is good}\}$. This implies
\begin{equation}\label{prob-1}
\mathds{P}[\text{the event in Definition \ref{def-vg} (1) occurs}]\geq \mathds{P}[(I_{i,j}^-,I_{i,j}^+)\text{ is good}].
\end{equation}
We use the event in Definition \ref{def-vg} (1) here because the effective resistance $\widehat{R}_n$ (see \eqref{def-Rn-2}) we mainly consider in this section does not depend on whether the intervals contained in $(-\infty,0)$ are good.
Specifically, the interval $(-\infty,0)$ serves as the outflow region for flows in the definition of $\widehat{R}_n$, preventing them from re-entering.

In addition, it is clear that the event in  Definition \ref{def-vg} (2) (resp. (3)) depends only on those edges with at least one endpoint falling within the interval $I_{i,j}^+$ (resp. $I_{i,j}^-$),
while the event in  Definition \ref{def-vg}  (1) depends only on the edge set
$$
\mathcal{E}_{[(z_i-2^{j+1})\vee 0,z_i+2^{j}]\times(z_i+2^{j+1},\infty)}
\cup\mathcal{E}_{[z_i-2^{j},z_i+2^{j+1}]\times{(0,(z_i-2^{j+1})\vee0)}}.
$$
Thus given $\mathcal{E}_n$, the events in  Definition \ref{def-vg}  (1), (2) and (3) are independent. Moreover, recall that $a_{\alpha,j}$ represents the $(1-\alpha)$-quantile of $\widehat{R}_j=\widehat{R}((-\infty,0],(2^j,+\infty))$ defined in \eqref{def-epsi}. Then from Lemma \ref{mono} (2) and the translation invariance of the model, we have
\begin{equation}\label{prob-2}
\begin{aligned}
&\mathds{P}[\text{the event in  Definition \ref{def-vg}  (2) occurs}]\\
=&\mathds{P}[\text{the event in  Definition \ref{def-vg}  (3) occurs}]\geq 1-\alpha.
\end{aligned}
\end{equation}

For the sake of concise notation, we write
\begin{equation}\label{eta-bk}
\overline{\eta}_{b_{k-1},b_k}=\sum_{b_k<l\leq b_{k-1}}\eta_l\quad \text{and}\quad \mu_{b_{k-1},b_k}=\mathds{E}[\overline{\eta}_{b_{k-1},b_k}]=\sum_{b_k<l\leq b_{k-1}}\mu_l \quad\text{for all }\quad k\geq 1.
\end{equation}
We next define the following event.

\begin{definition}\label{def-E}
For $\alpha\in(0,1/2]$ and $n\in \mathds{N}$, let $E_{\alpha,n}$ be the event that the following conditions hold.
\begin{itemize}
\item[(1)] For each $k\in[1,K_n]_{\mathds{Z}}$, $\overline{\eta}_{b_{k-1},b_k}\leq 2 \mu_{b_{k-1},b_k}$.

\medskip
\item[(2)] For each $k\in[1,K_n]_{\mathds{Z}}$ and each $i\in (\overline{\eta}_{b_k},\overline{\eta}_{b_0}]_{\mathds{Z}}$, there exists at least one $\alpha$-very good pair of intervals in $\{(I^-_{i,j},I^+_{i,j}):\ j\in (b_{k},b_{k-1}]_{\mathds{Z}}\}$.

\end{itemize}
\end{definition}

\begin{figure}[htbp]
\centering
\includegraphics[height=6cm,width=15cm]{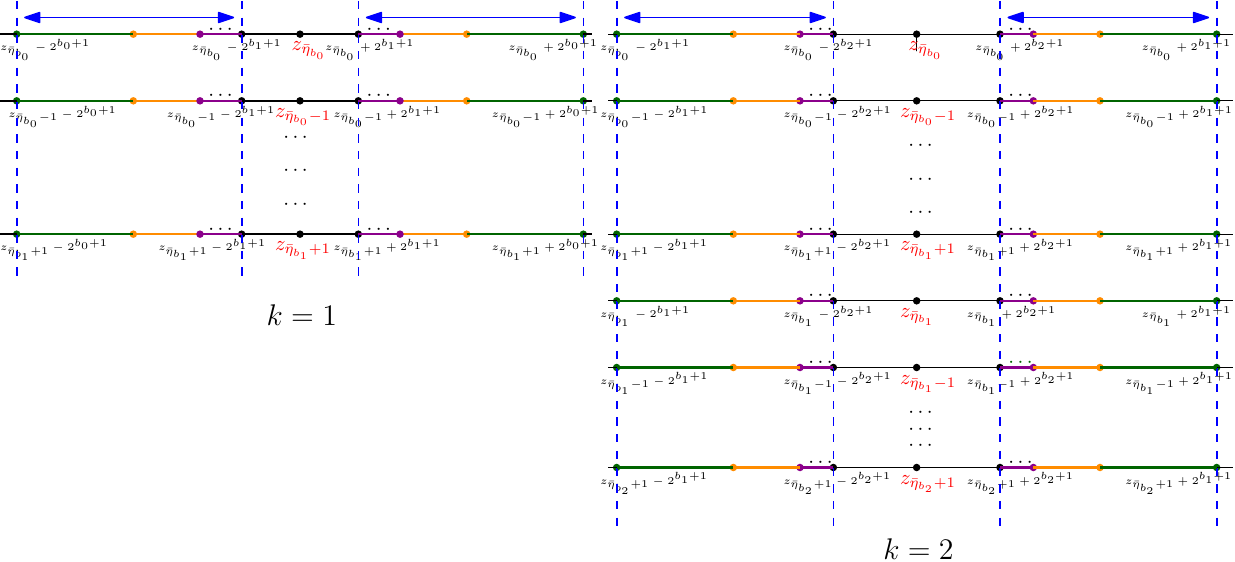}
\caption{The illustration for Definition \ref{def-E} (2) for $k=1$ (on the left) and $k=2$ (on the right). When $k=1$, we search for $\alpha$-very good pairs of intervals among $(I^-_{i,j},I^+_{i,j})$ for $i\in (\overline{\eta}_{b_1},\overline{\eta}_{b_0}]_{\mathds{Z}}$ and $j\in (b_{1},b_{0}]_{\mathds{Z}}$, which are subintervals (marked with green, orange and magenta colors) between blue dashed lines. When $k=2$, we search for $\alpha$-very good pairs of intervals among $(I^-_{i,j},I^+_{i,j})$ for $i\in (\overline{\eta}_{b_2},\overline{\eta}_{b_0}]_{\mathds{Z}}$ and $j\in (b_{2},b_{1}]_{\mathds{Z}}$,  which are subintervals (marked with green, orange and magenta colors) between blue dashed lines.}
\label{explore}
\end{figure}

The main result of this subsection provides the following estimate for $\mathds{P}[E_{\alpha,n}]$.

 \begin{proposition}\label{Prob-vg}
For any $\varepsilon\in(0,1)$, there exist $L>0$ large enough {\rm(}depending only on $\beta${\rm)}, $M=M(\varepsilon)>0$ and $n_0=n_0(\varepsilon)>0$ {\rm(}depending only on $\beta$ and $\varepsilon${\rm)} such that for all $\alpha\in(0,1/2]$ and all $n>n_0$,  we have $\mathds{P}[E_{\alpha,n}]>1-\varepsilon/2$.
\end{proposition}

In what follows, we fix $\alpha\in (0,1/2]$. To simplify notation, we will use $E_{n}$ to represent $E_{\alpha,n}$.
 Since the probability of the event $E_{n}$ is challenging to estimate directly, we will decompose it further as follows.
\begin{definition}\label{def-Eki}
For each $k\in [1,K_n]_{\mathds{Z}}$ and each $i\in (\overline{\eta}_{b_k},\overline{\eta}_{b_0}]_{\mathds{Z}}$, let $E_{n,k,i}$ be the event that none of pairs of intervals $(I_{i,j}^-,I_{i,j}^+)$, $j\in (b_{k},b_{k-1}]_{\mathds{Z}}$ is $\alpha$-very good.
\end{definition}
Then from the definition of $E_n$ in Definition \ref{def-E}, we can see that
\begin{equation}\label{decomp-En}
\text{the event in Definition \ref{def-E} (2)}=\cap_{k=1}^{K_n}\cap_{i=\overline{\eta}_{b_k}+1}^{\overline{\eta}_{b_0}}E^c_{n,k,i}.
\end{equation}
We now  provide an upper bound for $\mathds{P}[E_{n,k,i}]$.

\begin{lemma}\label{pro-ki}
There exists a constant $c_1=c_1(\beta)>0$ depending only on $\beta$ such that for each $k\in [1,K_n]_{\mathds{Z}}$ and each $i\in (\overline{\eta}_{b_k},\overline{\eta}_{b_0}]_{\mathds{Z}}$, we have
$$
\mathds{P}[E_{n,k,i}]\leq \exp\{-c_1(b_{k-1}-b_k)\}.
$$
\end{lemma}

\begin{proof}
For each $k\in [1,K_n]_{\mathds{Z}}$ and each $i\in (\overline{\eta}_{b_k},\overline{\eta}_{b_0}]_{\mathds{Z}}$, we let $A_{k,i}$ be the event that there exist at least $(1-\kappa)(b_{k-1}-b_k)/2$ pairs of intervals in $\{(I_{i,j}^-,I_{i,j}^+):\ j\in (b_k,b_{k-1}]_{\mathds{Z}}\}$ such that the event in Definition \ref{def-vg} (1) occurs. Here $\kappa$ is the constant in Proposition \ref{ld-xi}, depending only on $\beta$.
Note that the event $A_{k,i}$ is determined by the edge set $\mathcal{E}_{(0,+\infty)^2}$ and the position of $z_i$ from the Definition \ref{def-vg} (1). Moreover, the position of $z_i$ is determined by $\mathcal{E}_n=\mathcal{E}_{(-\infty,0]\times (0,N]}$, which is independent of  $\mathcal{E}_{(0,+\infty)^2}$.
Thus, from Proposition \ref{ld-xi} and \eqref{prob-1} we get that
\begin{equation*}
\mathds{P}[A_{k,i}^c|\mathcal{E}_n]\leq \exp\{-(1-\kappa)^2(b_{k-1}-b_k)/2\}.
\end{equation*}
By taking expectations on both sides of the above inequality, we obtain
\begin{equation}\label{prob-A}
\mathds{P}[A_{k,i}^c]\leq \exp\{-(1-\kappa)^2(b_{k-1}-b_k)/2\}.
\end{equation}

Additionally, recall that given $\mathcal{E}_n$, the events in (1), (2) and (3) of Definition \ref{def-vg} are independent. Therefore, combining this with \eqref{prob-A}, \eqref{prob-2} and $\alpha\in (0,1/2]$,  we obtain that there exists a constant $c_1=c_1(\beta)>0$ depending only on $\beta$ such that
\begin{align*}
\mathds{P}[E_{n,k,i}]&\leq \mathds{P}[A_{k,i}^c]+\mathds{P}[A_{k,i}\cap E_{n,k,i}]
= \mathds{P}[A_{k,i}^c]+\mathds{E}[\mathds{P}[A_{k,i}\cap E_{n,k,i}|\mathcal{E}_n]]\\
&\leq \exp\{-(1-\kappa)^2(b_{k-1}-b_k)/2\}+{(1-(1-\alpha)^2)}^{-(1-\kappa)(b_{k-1}-b_k)/2}\\
&\leq \exp\{-(1-\kappa)^2(b_{k-1}-b_k)/2\}+(3/4)^{-(b_{k-1}-b_k)/2}\\
&\leq \exp\{-c_1(\beta)(b_{k-1}-b_k)\}.
\end{align*}
Hence the proof is complete. \end{proof}

In addition, for each $k\in [1,K_n]_{\mathds{Z}}$, we denote $D_k$ for the event that $\overline{\eta}_{b_{k-1},b_k}>2 \mu_{b_{k-1},b_k}$, and denote $D=\cup_{k\geq 1}D_k$.
It is clear that $D^c$ is the event in Definition \ref{def-E} (1).

According to the Chernoff bound, we have the following estimate.

\begin{lemma}\label{Poitail-lem}
For each $k\geq 1$,
\begin{equation*}
\mathds{P}[D_k]\leq \exp\left\{-\mu_{b_{k-1},b_k}/3\right\}.
\end{equation*}
Therefore,
\begin{equation}\label{Poi-tail}
\mathds{P}[D]\leq \sum_{k\geq1}\exp \left\{-\mu_{b_{k-1},b_k}/3\right\}.
\end{equation}
\end{lemma}
\begin{proof}
For each $k\in [1,K_n]_{\mathds{Z}}$, by the Chernoff bound and definitions of $\overline{\eta}_{b_{k-1},b_k}$ and $\mu_{b_{k-1},b_k}$ in \eqref{eta-bk}, we get that
$$
\mathds{P}[D_k]=\mathds{P}[\overline{\eta}_{b_{k-1},b_k}>2\mu_{b_{k-1},b_k}]\leq \e^{-\frac{\mu_{b_{k-1},b_k}}{3}}.
$$
By taking the summation over $k$ on both sides of the above inequality, we obtain \eqref{Poi-tail}.
\end{proof}

To prove Proposition \ref{Prob-vg}, we also need  some asymptotic properties of the sequence $\{b_k\}_{k\geq 0}$. To this end, we define $d_k=n-b_k$ for $k\geq 0$. That is,
\begin{equation}\label{def-ck}
d_0=0\quad \text{and}\quad d_k=d_{k-1}+M+L\log(M+d_{k-1})\quad \text{for }k\geq 1.
\end{equation}

\begin{lemma}\label{ck-order-lem}
Let $L>2$ and $M>2$. For the sequence $\{d_k\}_{k\geq 0}$, we have
\begin{equation}\label{ck-order}
kM+L\sum_{i=1}^k\log i\leq d_k\leq kM+2L\sum_{i=1}^k \log(LMi).
\end{equation}
\end{lemma}
\begin{proof}
From \eqref{def-ck}, we can see that $\{d_k\}_{k\geq 0}$ is increasing and is just determined by $L$ and $M$.
We will establish \eqref{ck-order} by  an induction on $k$.

Since $d_0=0$ and $d_1=M+L\log M$, it is clear that \eqref{ck-order} holds when $k=0,1$.
Now assume that \eqref{ck-order} holds for some $k\geq 1$. For the lower bound, we get that
\begin{equation*}
	\begin{aligned}
		d_{k+1}&=d_k+M+L\log(M+d_k)\\
		&\geq (k+1)M+L\sum_{i=1}^k\log i+ L\log\left((k+1)M+\sum_{i=1}^k \log i\right)\\
		&\geq (k+1)M+L\sum_{i=1}^{k+1}\log i,
	\end{aligned}
\end{equation*}
which implies that the lower bound in \eqref{ck-order} holds for $k+1$.

We now turn to the upper bound. Since \eqref{ck-order} holds for $k$, we get that
\begin{equation}\label{logck-2}
	\begin{aligned}
		\log(M+d_k)&\leq \log\left((k+1)M+2L\sum_{i=1}^k \log(LMi)\right)\\
		&\leq \log\bigg((k+1)M+2Lk \log(LMk)\bigg)\\
		&=2\log\left(LM(k+1)\right)+\log\frac{(k+1)M+2Lk \log(LMk)}{L^2M^2(k+1)^2}\\
		&\leq 2\log\left(LM(k+1)\right).
	\end{aligned}
\end{equation}
Here the last inequality holds because $L^2M^2(k+1)^2>(k+1)M+2Lk\log(LMk)$ when $k\geq 1,\ L>2$ and $M>2$. Then applying \eqref{logck-2} to \eqref{def-ck}, we get that
\begin{equation*}
	\begin{aligned}
	d_{k+1}&=d_k+M+L\log(M+d_k)\\
	&\leq (k+1)M+2L\sum_{i=1}^k \log (LMi)+2L\log\left(LM(k+1)\right)\\
	&= (k+1)M+2L\sum_{i=1}^{k+1} \log (LMi).
	\end{aligned}
\end{equation*}
Consequently, \eqref{ck-order} holds for all $k\geq 0$ by induction.
\end{proof}

From \eqref{def-ck} and Lemma \ref{ck-order-lem} we obtain that
\begin{equation}\label{bk-1-bk}
b_{k-1}-b_k=M+L\log(M+d_{k-1})\in\left[M+L\log k,\ M+2L\log(LMk)\right].
\end{equation}

With the above lemmas at hand, we can present the

\begin{proof}[Proof of Proposition \ref{Prob-vg}]
From \eqref{decomp-En}, \eqref{Poi-tail} and Lemma \ref{pro-ki}, we arrive at
\begin{equation}\label{PEc}
\begin{aligned}
\mathds{P}[E_n^c]&\leq \mathds{P}\left[D\cup \left(\cup_{k=1}^{K_n}\cup_{i=\bar{\eta}_{b_k}+1}^{\overline{\eta}_{b_0}}E_{n,k,i}\right)\right]
=\mathds{P}[D]+\mathds{P}\left[D^c\cap \left(\cup_{k=1}^{K_n}\cup_{i=\overline{\eta}_{b_k}+1}^{\overline{\eta}_{b_0}}E_{n,k,i}\right)\right]\\
&\leq  \sum_{k\geq1}\exp\{-\mu_{b_{k-1},b_k}/3\} +\sum_{k\geq 1} \left(2\sum_{i=1}^{k}\mu_{b_{i-1},b_i}\right)\exp\{-c_1(b_{k-1}-b_k)\}.\\
\end{aligned}
\end{equation}
In addition, it follows from \eqref{bk-1-bk} and \eqref{mu-i} that there exist two constants $c_2,c_3>0$ (depending only on $\beta$) such that for each $i\geq 1$,
\begin{equation}\label{mui}
 c_2(M+L\log i)\leq \mu_{b_{i-1},b_i}=\sum_{b_i<l\leq b_{i-1}}\mu_l\leq c_3(M+L\log(LMi)).
\end{equation}
This implies
\begin{equation}\label{sum-mu}
2\sum_{i=1}^{k}\mu_{b_{i-1},b_i}\leq c_4ML(\log L) k\log k
\end{equation}
for some constant $c_4>0$ depending only on $\beta$.
Applying \eqref{bk-1-bk}, \eqref{mui} and \eqref{sum-mu} to \eqref{PEc}, we conclude that
\begin{align*}
\mathds{P}[E_n^c]&\leq \e^{-c_2M}\sum_{k\geq 1}\e^{-c_2L\log k}+c_5ML(\log L)\e^{-c_6M}\sum_{k\geq 1}k\log k\e^{-c_1L\log k},
\end{align*}
where $c_5,c_6>0$ are some constants depending only on $\beta$. We now take $L$ large enough such that
\begin{equation}\label{cond-L}
L\geq \max\left\{2,3c_1^{-1},3c_2^{-1}\right\}.
\end{equation}
It is worth emphasizing that constants $c_1$ and $c_2$ depend only on $\beta$, meaning that $L$ also depends only on $\beta$. Then we get
\begin{equation}\label{E-M}
\mathds{P}[E_n^c]\leq c_7\e^{-c_8M}
\end{equation}
for some constants $c_7,c_8>0$ depending only on $\beta$. Hence we complete the proof by taking $M\geq c_8^{-1}\log(2c_7/\varepsilon)$ and sufficiently large $n$.
\end{proof}

\subsection{A recursive formula of the $(1-\alpha)$-quantile $a_{\alpha,n}$}\label{Rn-sect}

Recall that for $\alpha\in(0,1/2]$ and $n\geq 1$, $\widehat{R}_n$ represents the effective resistance defined in \eqref{def-Rn-2}, and $a_{\alpha,n}$ defined in \eqref{def-epsi} represents the $(1-\alpha)$-quantile of $\widehat{R}_n$.
The following proposition is the main output of this subsection,
which gives a recursive relation for the sequence $\{a_{\alpha, i}\}_{i\geq 1}$.

\begin{proposition}\label{an-thm}
For any $\beta>0$, there exists a constant $c>0$ depending only on $\beta$ such that the following holds for all $\alpha\in(0,1/2]$.
There exist $L>0$ large enough {\rm(}depending only on $\beta${\rm)}, $M'=M'(\alpha)>0$ and $n'_0=n'_0(\alpha)>0$ {\rm(}depending only on $\beta$ and $\alpha${\rm)} such that for all $n>n'_0$,
\begin{equation*}\label{an-lb}
a_{\alpha, n}\geq c\sum_{{ k=2}}^{K_n}\frac{\min\{a_{\alpha,b_{k}+1},\cdots, a_{\alpha,b_{k-1}}\}}{d_k},
\end{equation*}
where $d_k=n-b_k$ is defined in \eqref{def-ck}.
\end{proposition}

The proof of Proposition \ref{an-thm} will be presented at the end of this subsection. Before that, we make some preparations. Let us start by controlling the position of $z_{\overline{\eta}}$, which will help us determine which $j$ will satisfy that $I_{\overline{\eta},j}^+\subset (0,N]$.

\begin{lemma}\label{z-eta}
For any $\gamma\in (0,1)$, there exist constants $c'=c'(\gamma)\in (0,1)$ and $n_1=n_1(\gamma)>0$ {\rm(}both depending only on $\beta$ and $\gamma${\rm)} such that for all $n\geq n_1$,
$$
\mathds{P}\left[z_{\overline{\eta}}\leq c'2^{n}\right]\geq 1-\gamma/2.
$$
We will refer to the event $\{z_{\overline{\eta}}\leq c'2^{n}\}$ as $F_{\gamma,n}$.
\end{lemma}

\begin{proof}
By the definitions  of $z_{\overline{\eta}}$ and $\xi_{uv}$ in \eqref{xi-uv}, we can see that
\begin{equation*}
\begin{aligned}
\mathds{P}\left[z_{\overline{\eta}}\leq c'2^{n}\right]&=\mathds{P}\left[\sum_{u=-\infty}^0\sum_{v=\lfloor c'2^{n}\rfloor+1}^{2^n}\xi_{uv}=0\right]\\
&=\exp\left\{-\beta\int_{-\infty}^1\int_{\lfloor c'2^{n}\rfloor+1}^{2^n+1}\frac{1}{|x-y|^2}\d x\d y\right\}\geq (c'-2^{-n})^\beta.
\end{aligned}
\end{equation*}
Hence, for fixed $\gamma\in (0,1)$, we can complete the proof by taking $c'\geq (1+(1-\gamma/2)^{1/\beta})/2$ and $n_1=\log(1/(1-(1-\gamma/2)^{1/\beta}/2))$.
\end{proof}

Define
\begin{equation}\label{k0}
\begin{aligned}
k_0=k_0(\gamma)&=\inf\left\{k\geq 1:\ b_{k-1}+1\leq n-\log_2(1/(1-c'))\right\}\\
&=\inf\left\{k\geq 1:\  d_{k-1}-1\geq  \log_2(1/(1-c'))\right\},
\end{aligned}
\end{equation}
where $c'=c'(\gamma)$ is the constant in Lemma \ref{z-eta}.
It can be observed that $k_0$ depends only on $\beta$ and $\gamma$, and the definition of it ensures that on the event $F_{\gamma, n}$ we have $$I_{\overline{\eta},b_{k_0-1}}^+=(z_{\overline{\eta}}+2^{b_{k_0-1}}, z_{\overline{\eta}}+2^{b_{k_0-1}+1}]\subset (0,2^n].$$

The key input of the proof of Proposition \ref{an-thm} is to show that on the event $E_{\alpha,n}\cap F_{\alpha,n}$, the effective resistance $\widehat{R}_n$ is bounded from below by $\{a_{\alpha,k}\}_{k\geq 1}$ as follows.

\begin{lemma}\label{induct-thm}
For any $\beta>0$, there exists a constant $c>0$ depending only on $\beta$ such that the following holds for all $\alpha\in(0,1/2]$. On the event $E_{\alpha, n}\cap F_{\alpha,n}$, we have
\begin{equation*}\label{Rn-lb}
\widehat{R}_n\geq c\sum_{k=k_0}^{K_n}\frac{\min\{a_{\alpha,b_{k}+1},\cdots, a_{\alpha,b_{k-1}}\}}{d_k},
\end{equation*}
where $K_n$ and $\{d_k\}_{k\geq 0}$ are defined in \eqref{def-Kn} and \eqref{def-ck}, respectively, and $k_0$ is the constant defined in \eqref{k0} with $\gamma=\alpha$.
\end{lemma}

To prove Lemma \ref{induct-thm}, we recall that $Z$ is the set defined in \eqref{def-Z}.
Let $\theta$ denote a unit flow that satisfies conditions specified in \eqref{def-Rn-2}, i.e., $\theta$ is a unit flow from $(-\infty,0]$ to $(2^n,+\infty)$ with $\theta_{uv}=0$ for all $\langle u,v\rangle\in \mathcal{E}_{(-\infty,0]\times (2^n,+\infty)}$.
For notation convenience, for $i\in [0,\overline{\eta}]_{\mathds{Z}}$, let $\theta_i$ denote the portion of the flow $\theta$ that passes through the point $z_i$. Therefore, $\theta_i$ is a flow that enters the interval $(0, 2^n]$ starting from the point $z_i$ (see Figure \ref{flow}).
More precisely, we can decompose the flow $\theta$ into self-avoiding paths $P$ using the algorithm in \cite[Page 54]{BDG20}, and thus obtain a sequence of flows $\theta_P$. Then we can see that
$$
\theta_i=\sum_{P:z_i\in P}\theta_P.
$$
Here $z_i\in P$ means that $z_i$ lies on the path $P$. It can be observed from the algorithm that $\theta_1,\cdots, \theta_{\overline{\eta}}$ exhibit unidirectionality, i.e.
\begin{equation}\label{unidirect}
(\theta_i)_{uv}(\theta_j)_{uv}\geq 0\quad \text{for all $i,j\in [1,\overline{\eta}]_{\mathds{Z}}$ and }\langle u,v\rangle\in \mathcal{E},
\end{equation}
and
\begin{equation}\label{cond-theta}
\sum_{i=0}^{\overline{\eta}} |\theta_i|=1\quad \text{and}\quad  \sum_{i=0}^{\overline{\eta}}\theta_i=\theta.
\end{equation}
Here, $|\cdot|$ refers to the amount of the flow as defined in \eqref{size}.

\begin{figure}[htbp]
\centering
\includegraphics[height=1.2cm,width=15cm]{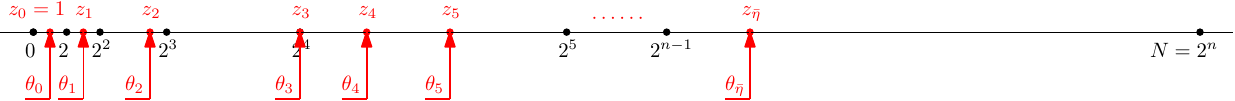}
\caption{The illustration for points $\{z_i\}_{i\geq 0}$ and flows $\{\theta_i\}_{i\geq 0}$. The red points correspond  to the set $Z=\{z_i\}_{i\geq 0}$, the collection of points through which flows enter $(0,2^n]$. The red arrows represent flows $\{\theta_i\}_{i\geq 0}$ entering at  points $\{z_i\}_{i\geq 0}$.}
\label{flow}
\end{figure}

 For each $k\in [k_0,K_n]_{\mathds{Z}}$, consider the point $z_{i_k}$, where $i_k:=\overline{\eta}_{b_k}+1$.
Indeed, by the definition of $\overline{\eta}_{b_k}$ in \eqref{def-etabar},  we have
$$
z_{i_k}=\inf\left\{z\in Z:\  z\in (2^{b_k+1},2^{b_k+2}]\right\}.
$$
We also define $j^{(k)}_{i_k}\in (b_k,b_{k-1}]$ as the number such that $(I_{i_k,j^{(k)}_{i_k}}^-,I_{i_k,j^{(k)}_{i_k}}^+)$ is an $\alpha$-very good pair of intervals. If there are multiple such numbers, we take $j^{(k)}_{i_k}$ to be the smallest one. It is worth emphasizing that $j^{(k)}_{i_k}$ is well defined on the event $E_{\alpha,n}$.
Additionally, for any flow $f$ and any interval $I\subset \mathds{Z}$, let us define $f(I)$ as the portion of the flow $f$ that passes through $I$. That is,
\begin{equation}\label{flow-int}
f(I)_{uv}=
\begin{cases}
f_{uv},\quad &\langle u,v\rangle\in \mathcal{E}_{I\times \mathds{Z}},\\
0,\quad &\text{otherwise}.
\end{cases}
\end{equation}
For $i\in [1,\overline{\eta}]_{\mathds{Z}}$, we also write $\theta_{\leq i}(I)$ for the portion of flows $\theta_1,\cdots, \theta_i$ that pass through $I$, i.e.,
\begin{equation}\label{flow-int-2}
(\theta_{\leq i}(I))_{uv}=
\begin{cases}
\sum_{m=1}^i(\theta_m)_{uv},\quad &\langle u,v\rangle\in \mathcal{E}_{I\times \mathds{Z}},\\
0,\quad &\text{otherwise}.
\end{cases}
\end{equation}

\begin{lemma}\label{merge}
Assume that the event $E_{\alpha,n}\cap F_{\alpha,n}$ occurs and recall $k_0$ is the constant defined in \eqref{k0} with $\gamma=\alpha$. Then for each $k\in [k_0,K_n]_{\mathds{Z}}$,
$$
|\theta_{\leq i_k}(I_{i_k,j^{(k)}_{i_k}}^+)|\geq |\theta_{i_k}|+\sum_{i=0}^{\overline{\eta}_{b_k}}|\theta_i|=\sum_{i=0}^{\overline{\eta}_{b_k}+1}|\theta_i|.
$$

\end{lemma}


\begin{proof}
Assume that $E_{\alpha,n}\cap F_{\alpha,n}$  occurs and fix $k\in [k_0,K_n]_{\mathds{Z}}$.  For notation convenience, we will denote $j^{(k)}_{i_k}$ as $j_{i_k}$ throughout the proof. By Definition \ref{def-vg} (1) for the $\alpha$-very good pair of intervals, we can see that on the event $E_{\alpha,n}$, we have
\begin{itemize}

\item[(1)] $[(z_{i_k}-2^{j_{i_k}+1})\vee 0,\ z_{i_k}+2^{j_{i_k}}]$ and $(z_{i_k}+2^{j_{i_k}+1},+\infty)$ are not directly connected by any long edge;

\item[(2)] $(0,\ (z_{i_k}-2^{j_{i_k}+1})\vee 0)$ and $[z_{i_k}-2^{j_{i_k}},z_{i_k}+2^{j_{i_k}+1}]$ are not directly connected by any long edge.
\end{itemize}
Moreover, since $z_{i_k}\in (2^{b_k+1},2^{b_k+2}]$ and $j_{i_k}\in (b_k,b_{k-1}]$, it is clear that
\begin{equation}\label{contain}
[0,2^{b_k+1}]\subset [(z_{i_k}-2^{j_{i_k}+1})\vee 0,\ z_{i_k}+2^{j_{i_k}}]=[0,z_{i_k}+2^{j_{i_k}}].
\end{equation}
In addition, by  definitions of $F_{\alpha,n}$ and $k_0$ in \eqref{k0}, we have that
$$
I_{i,j}^+\subset (0,2^n]\quad \text{for all }i\in [1,\overline{\eta}]_{\mathds{Z}}\ \text{and }j\in[1, b_{k_0-1}]_{\mathds{Z}},
$$
which implies $I_{i_k,j_{i_k}}^+\subset (0,2^n]$ for all $k\in [k_0,K_n]_{\mathds{Z}}$.
Combining this with \eqref{contain} and (1), we conclude that every flow, departing from the interval $[0,2^{b_k+1}]\subset [0,z_{i_k}+2^{j_{i_k}}]$ to $(2^n,+\infty)$, must pass through the interval $I_{i_k,j_{i_k}}^+$. This implies that the flow $\theta_i$ passes through $I_{i_k,j_{i_k}}^+$ for all $0\leq i \leq \overline{\eta}_{b_k}$.
Similarly, from (1) and $I_{i_k,j_{i_k}}^+\subset (0,2^n]$ again, we can also see that the flow $\theta_{i_k}$ (from $z_i$ to $(2^n,+\infty)$) must pass through  the interval $I_{i_k,j_{i_k}}^+$ (see Figure \ref{merge-f} for an illustration).
Therefore, combining this with the definition of $\theta_{\leq i_k}(I_{i_k,j_{i_k}}^+)$ in \eqref{flow-int-2}, we can obtain the desired result.
\end{proof}

\begin{figure}[htbp]
\centering
\includegraphics[height=1.6cm,width=16cm]{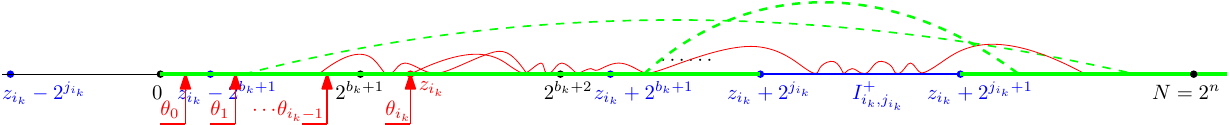}
\caption{The illustration for the proof of Lemma \ref{merge} (note that we also write $j_{i_k}$ for $j_{i_k}^{(k)}$ in this illustration).  The blue line represents the interval $I_{i_k,j_{i_k}}^+$, situated to the left of the point $N=2^n$.
The red arrows represent flows entering the interval $[0,2^{b_k+1}]$ and the point $z_{i_k}$. The
green dashed lines in the graph represent the absence of long edges directly connecting intervals $[(z_{i_k}-2^{j_{i_k}+1})\vee 0,\ z_{i_k}+2^{j_{i_k}}]$  and $[z_{i_k}+2^{j_{i_k}+1},+\infty]$ (the green lines). Thus flows, departing from the interval $[0,2^{b_k+1}]$ and the point $z_{i_k}$ to $(2^n,+\infty)$, must pass through $I_{i_k,j_{i_k}}^+$ (the blue line).}
\label{merge-f}
\end{figure}

In the following, we will use the $j$-th layer pairs of intervals to represent $(I_{i,j}^-,I_{i,j}^+)$ for $i\in [0,\overline{\eta}]$.
For each $k\in [k_0,K_n]_{\mathds{Z}}$, we want to establish a lower bound for the amount of flows passing through the pairs of intervals from the $(b_k+1)$-th to the $b_{k-1}$-th layers. 

For fixed $k\in[k_0,K_n]_{\mathds{Z}}$, let us start by extending the definition of $j^{(k)}_{i_k}$ to general $i$.
For each $i\in(\overline{\eta}_{b_k},\overline{\eta}_{b_0}]$, let us define $j^{(k)}_i\in (b_k,b_{k-1}]$ as the number such that $(I_{i,j^{(k)}_i}^-,I_{i,j^{(k)}_i}^+)$ is an $\alpha$-very good pair of intervals. If there are multiple such numbers, we select the smallest one for $j^{(k)}_i$. 
Additionally, it is important to emphasize that when $i=\overline{\eta}_{b_k}+1$, the $\alpha$-very good pair of intervals chosen here corresponds to $(I_{i_k,j^{(k)}_{i_k}}^-,I_{i_k,j^{(k)}_{i_k}}^+)$, which is defined in the paragraph below \eqref{cond-theta}.


We show that flows passing through those $\alpha$-very good pairs of intervals $(I_{i,j^{(k)}_i}^-,I_{i,j^{(k)}_i}^+)$, $i\in(\overline{\eta}_{b_k},\overline{\eta}_{b_0}]$ are significant as follows.

\begin{lemma}\label{lb-k}
For any fixed $k\in[k_0,K_n]_{\mathds{Z}}$, we have
$$
|\theta_{\leq i_k}(I_{i_k,j^{(k)}_{i_k}}^+)|
+\sum_{i=\overline{\eta}_{b_k}+2}^{\overline{\eta}}\max\left\{|\theta_i(I_{i,j^{(k)}_i}^-)|, |\theta_i(I_{i,j^{(k)}_i}^+)|\right\}\geq 1/2.
$$
\end{lemma}


\begin{proof}
For each $i\in [\overline{\eta}_{b_k}+2,\overline{\eta}_{b_0}]_{\mathds{Z}}$, since $(I_{i,j^{(k)}_i}^-,I_{i,j^{(k)}_i}^+)$ is an $\alpha$-very good pair of intervals, the flow $\theta_i$ emanating from the point $z_i$ to $(2^n,+\infty)$ must pass through the interval $I_{i,j^{(k)}_i}^+$ or $I_{i,j^{(k)}_i}^-$. This means that
$$
|\theta_i(I_{i,j^{(k)}_i}^-)|+|\theta_i(I_{i,j^{(k)}_i}^+)|\geq |\theta_i|.
$$
Therefore, $\max\{|\theta_i(I_{i,j^{(k)}_i}^-)|, |\theta_i(I_{i,j^{(k)}_i}^+)| \}\geq |\theta_i|/2$. Combining this with Lemma \ref{merge}, \eqref{cond-theta} and $\overline{\eta}_{b_0}=\overline{\eta}$, we arrive at
$$
\begin{aligned}
&|\theta_{\leq i_k}(I_{i_k,j^{(k)}_{i_k}}^+)|
+\sum_{i=\overline{\eta}_{b_k}+2}^{\overline{\eta}_{b_0}}\max\left\{|\theta_i(I_{i,j^{(k)}_i}^-)|, |\theta_i(I_{i,j^{(k)}_i}^+)| \right\}\\
& \geq \sum_{i=0}^{\overline{\eta}_{b_k}+1}|\theta_i|+\frac{1}{2}\sum_{i=\overline{\eta}_{b_k}+2}^{\overline{\eta}}|\theta_i|\geq \frac{1}{2}.
\end{aligned}
$$
Hence, the proof is complete.
\end{proof}

 We also need the following lemma.
\begin{lemma}\label{i12}
Assume that $k_1,k_2\in [k_0,K_n]_{\mathds{Z}}$ with $k_1\neq k_2$. If $I^{\pm}_{i_1,j_{i_1}^{(k_1)}}\cap I^{\pm}_{i_2,j_{i_2}^{(k_2)}}\neq \varnothing$ or $I^{\pm}_{i_1,j_{i_1}^{(k_1)}}\cap I^{\mp}_{i_2,j_{i_2}^{(k_2)}}\neq \varnothing$, then we have $i_1\neq i_2$.
\end{lemma}

\begin{proof}
We will prove the case when $I^{+}_{i_1,j_{i_1}^{(k_1)}}\cap I^{+}_{i_2,j_{i_2}^{(k_2)}}\neq \varnothing$ here, and the other cases can be proved similarly. Given that $k_1\neq k_2$, from $j_\cdot^{k}\in (b_k,b_{k-1}]$ for all $k\in [k_0,K_n]_{\mathds{Z}}$ we have
\begin{equation}\label{k12}
j_{i_1}^{(k_1)}\neq j_{i_2}^{(k_2)}.
\end{equation}

To complete the proof, we will proceed by a contradiction. Let us assume that $i_1=i_2$. Then from this assumption and the definition of $I_{ij}^+$ in \eqref{I+-}, we get
\begin{equation*}\label{varno}
(z_{i_1}+2^{j_{i_1}^{(k_1)}},z_{i_1}+2^{j_{i_1}^{(k_1)}+1}]\cap  (z_{i_1}+2^{j_{i_1}^{(k_2)}},z_{i_1}+2^{j_{i_1}^{(k_2)}+1}]\neq \varnothing,
\end{equation*}
This implies $j_{i_1}^{(k_1)}= j_{i_1}^{(k_2)}$, which contradicts \eqref{k12}.
\end{proof}

We now turn to the
\begin{proof}[Proof of Lemma \ref{induct-thm}]
For a fixed $\alpha\in(0,1/2]$, recall that we have assumed that the event $E_{\alpha,n}\cap F_{\alpha,n}$ occurs,
$k_0=k_0(\alpha)$ is the constant defined in \eqref{k0} with $\gamma=\alpha$,
and $a_{\alpha,i}$ represents the $(1-\alpha)$-quantile of the effective resistance $\widehat{R}_n$ (see \eqref{def-epsi}).

For each $k\in[k_0,K_n]_{\mathds{Z}}$, since the event in Definition \ref{def-E} (1) for $E_{\alpha,n}$ occurs, from \eqref{eta-bk} and \eqref{mu-i} we can see that there is a constant $c_1(\beta)>0$ (depending only on $\beta$) such that
\begin{equation*}\label{eta}
\overline{\eta}_{b_0}-\overline{\eta}_{b_k}\leq 2 \sum_{i=1}^k\mu_{b_{i-1},b_i}\leq c_1(\beta) d_k \quad \text{for all }k\in [1,K_n]_{\mathds{Z}}.
\end{equation*}
Combining this with Lemma \ref{lb-k} and the definition of $\alpha$-very good pair of intervals in Definition \ref{def-vg} (2) and (3), we get that the sum of energies (i.e.\ effective resistance) of flows $\theta_0,\cdots,\theta_{\overline{\eta}}$ passing through the pairs of intervals from the $(b_k+1)$-th to the $b_{k-1}$-th layers is at least
\begin{equation}\label{k-layer}
\begin{aligned}
&|\theta_{\leq i_k}(I_{i_k,j^{(k)}_{i_k}}^+)|^2a_{\alpha, j^{(k)}_{i_k}}
+\sum_{i=\overline{\eta}_{b_k}+2}^{\overline{\eta}_{b_0}}\max\left\{|\theta_i(I_{i,j^{(k)}_i}^-)|, |\theta_i(I_{i,j^{(k)}_i}^+)| \right\}^2 a_{\alpha,j^{(k)}_i}\\
&\geq \min\{a_{\alpha,b_k+1},\cdots,a_{\alpha,b_{k-1}}\}\left(|\theta_{\leq i_k}(I_{i_k,j^{(k)}_{i_k}}^+)|^2+\sum_{i=\overline{\eta}_{b_k}+2}^{\overline{\eta}_{b_0}}\max\left\{|\theta_i(I_{i,j^{(k)}_i}^-)|, |\theta_i(I_{i,j^{(k)}_i}^+)| \right\}^2\right)\\
&\geq \frac{ \min\{a_{\alpha,b_k+1},\cdots,a_{\alpha,b_{k-1}}\}}{4c_1(\beta) d_k}.
\end{aligned}
\end{equation}

We next consider the total energy generated by flows passing through all layers from $k_0$ to $K_n$. It follows from Lemma \ref{i12} that although intervals $I^{\pm}_{i,j_i^{(k)}}$ from different layers may intersect, any two intersecting intervals must have different subscripts $i$.
This implies that for intersecting intervals, the energy considered in different layers comes from energies generated by different flows (see Figure \ref{intersection} for an illustration). Therefore, by \eqref{unidirect} we can add up the above energy of each layer and get a lower bound

\begin{figure}[htbp]
\centering
\includegraphics[height=2.0cm,width=12cm]{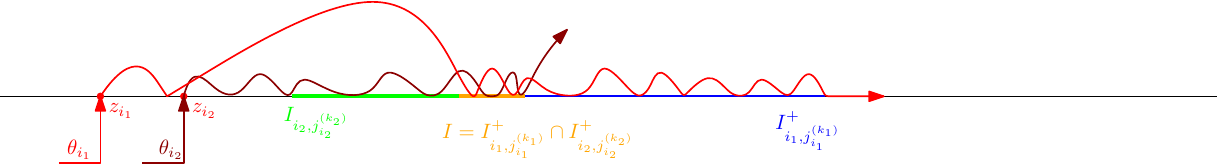}
\caption{The illustration for the intersection of $\alpha$-very good intervals from different layers. The red arrow represents the flow $\theta_{i_1}$  entering at the point $z_{i_1}$, while the dark red arrow represents the flow $\theta_{i_2}$ entering at the point $z_{i_2}$. The blue line represents the $\alpha$-very good interval $I^+_{i_1,j^{(k_1)}_{i_1}}$ discovered during the search in the $k_1$-th layer near $z_{i_1}$, while the green line represents the $\alpha$-very good interval $I^+_{i_2,j^{k_2}_{i_2}}$ discovered during the search in the $k_2$-th ($k_2< k_1$) layer near $z_{i_2}$. The orange line represents the interval $I:=I^+_{i_1,j_{i_1}^{(k_1)}}\cap I^+_{i_2,j^{k_2}_{i_2}}$.
It can be observed that in the $k_1$-th layer the energy in $I\subset I^+_{i_1,j_{i_1}^{(k_1)}}$ considered in \eqref{k-layer} (with $k=k_1$) is generated by the flow $\theta_{i_1}$, while in the $k_2$-th layer the energy in $I\subset I^+_{i_2,j_{i_2}^{(k_2)}}$ considered in \eqref{k-layer} (with $k=k_2$) is generated by the flow $\theta_{i_2}$.
It is important to point out that the total energy generated when flows $\theta_{i_1}$  and $\theta_{i_2}$ pass through the interval $I$ together is greater than the sum of energies generated when each of them individually passes through that interval, thanks to \eqref{unidirect}.
}
\label{intersection}
\end{figure}
\begin{equation}\label{sum-Rn}
\begin{aligned}
\widehat{R}_n
&\geq \frac{1}{4c_1(\beta)}\sum_{k=k_0}^{K_n}\frac{ \min\{a_{\alpha,b_k+1},\cdots,a_{\alpha,b_{k-1}}\}}{ d_k}.
\end{aligned}
\end{equation}
\begin{figure}[htbp]
\centering
\includegraphics[height=1.5cm,width=15cm]{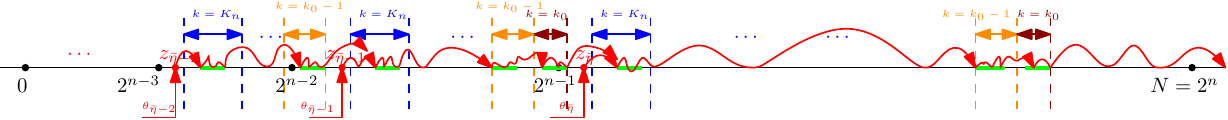}
\caption{The illustration for \eqref{sum-Rn}. The red arrows represent flows entering $(0,N]$; the blue, orange and dark red dashed lines represent the search steps.
In addition, green lines represent the $\alpha$-very good intervals discovered (note that, for clarity, the figure here only shows the case where flows pass through and exit from $I^+_{\cdot}$; in reality, flows may also pass through and exit from $I^-_{\cdot}$).
Each flow, as it moves rightward, must enter the green lines within the search range at each step and generate a certain amount of energy. Although these flows may intersect, the total energy generated by these flows is greater than the sum of energies generated by individual flows, thanks again to \eqref{unidirect}.
}
\label{flowsum}
\end{figure}
(See Figure \ref{flowsum} for an illustration). Hence we complete the proof by taking $c=1/(4c_1(\beta) )$, which depends only on $\beta$.
\end{proof}

We now can provide the

\begin{proof}[Proof of Proposition \ref{an-thm}]
Fix $\alpha\in(0,1/2]$. Throughout the proof, we also fix a sufficiently large $L$ satisfying \eqref{cond-L}, which depends only on $\beta$.

 According to \eqref{k0} and the definition of $b_k$ in \eqref{def-bk}, there exists a constant $M_1=M_1(\alpha)>0$ (depending only on $\beta$ and $\alpha$) such that for all $M\geq M_1$, we have
$
k_0=k_0(\alpha)=2.
$
In addition, by taking $\varepsilon=\alpha $ in Proposition \ref{Prob-vg}, we get that there exist $M_0=M_0(\alpha)>0$ and $n_0=n_0(\alpha)>0$ such that when $M\geq M_0$, we have $\mathds{P}[E_{\alpha,n}]\geq 1-\alpha/2$ for all $n> n_0$. Combining the above analysis with Lemma \ref{z-eta}, we can see that for $M'(\alpha)=\max\{M_1,M_0\}$ and for all $n\geq n_0'(\alpha):=\max\{n_0,n_1\}$,
\begin{equation*}\label{k0-1}
\mathds{P}[E_{\alpha,n}\cap F_{\alpha,n}]\geq 1-\alpha\quad \text{and}\quad  k_0=2.
\end{equation*}
Applying this into Lemma \ref{induct-thm} we arrive at
$$
\mathds{P}\left[\widehat{R}_n\geq c\sum_{k=2}^{K_n}\frac{\min\{a_{\alpha,b_{k}+1},\cdots, a_{\alpha,b_{k-1}}\}}{d_k}\right]\geq 1-\alpha
$$
for some constant $c>0$ depending only on $\beta$. Therefore, by the definition of $(1-\alpha)$-quantile of $\widehat{R}_n$ in \eqref{def-epsi}, we see that
$$
a_{\alpha,n}\geq c\sum_{k=2}^{K_n}\frac{\min\{a_{\alpha,b_{k}+1},\cdots, a_{\alpha,b_{k-1}}\}}{d_k}
$$
for all $n\geq n'_0$.
\end{proof}

\subsection{Proof of Theorem \ref{Rn} and Proposition \ref{quantile-lb}}\label{proof-mr1}

Let us start with the

\begin{proof}[Proof of Proposition \ref{quantile-lb}]
Fix $\alpha\in(0,1/2]$. Recall that Proposition \ref{an-thm} gives a recursive formula for $(1-\alpha)$-quantiles $\{a_{\alpha,i}\}_{i\geq 1}$. With the choice of $L$ (depending only on $\beta$) and $M'(\alpha)$ (depending only on $\beta$ and $\alpha$) as in Proposition \ref{an-thm}, from Lemma \ref{ck-order-lem} we can see that there exists $M_1\in \mathds{N}$ (depending only on $\beta$ and $\alpha$) such that
\begin{equation}\label{M1}
c\sum_{k=2}^{M_1}\frac{1}{d_k}>2.
\end{equation}
 We also choose $\delta>0$ (depending only on $\beta$ and $\alpha$) such that $\exp\{\delta d_{M_1}\}<2$, and choose $\widetilde{c}=\widetilde{c}(\alpha)>0$ (depending only on $\beta$ and $\alpha$) such that
\begin{equation}\label{a1-cM1}
a_{\alpha,k}\geq \widetilde{c}\e^{\delta k}\quad \text{for all }k\in [1,d_{M_1}]_{\mathds{Z}}.
\end{equation}

We now use induction to show that for all $n\geq 1$,
\begin{equation}\label{induct-an}
a_{\alpha,n}\geq \widetilde{c}\e^{\delta n}.
\end{equation}
The case for $n\in [1,d_{M_1}]_{\mathds{Z}}$ holds as in \eqref{a1-cM1} thanks to our choice of parameters.
Assume that $n>d_{M_1}$ and  \eqref{a1-cM1} holds for all $k\in [1,n-1]_{\mathds{Z}}$.
Then by Proposition \ref{an-thm}, \eqref{M1} and $b_k=n-d_k$, we obtain that
\begin{align*}
a_{\alpha,n}&\geq c\sum_{k=2}^{K_n}\frac{\min\{a_{\alpha,n-d_k+1},\cdots,a_{\alpha,n-d_{k-1}}\}}{d_k}
\geq c\sum_{k=2}^{K_n} \frac{\widetilde{c}\e^{\delta(n-d_k)}}{d_k}
\geq c\sum_{k=2}^{M_1}\frac{\widetilde{c}\e^{\delta(n-d_{M_1})}}{d_k}\\
&\geq \widetilde{c}\e^{\delta n}\left(c\sum_{k=2}^{M_1}\frac{1}{d_k}\right)\e^{-\delta d_{M_1}}\geq \widetilde{c}\e^{\delta n}.
\end{align*}
Consequently, \eqref{induct-an} holds for all $n\geq 1$.
\end{proof}

We now turn to the

\begin{proof}[Proof of Theorem \ref{Rn}]
Fix $\varepsilon\in (0,1/2]$. We begin by considering the event $E_{1/2,n}$.
From Proposition \ref{Prob-vg}, we obtain that there exist $L>0$ large enough (depending only on $\beta$), $M_0=M_0(\varepsilon)>0$ and $n_0=n_0(\varepsilon)>0$ (both depending only on $\beta$ and $\varepsilon$) such that when $M\geq M_0$, we have $\mathds{P}[E_{1/2,n}]\geq 1-\varepsilon/2$ for all $n\geq n_0$.

We next turn to the event $F_{\varepsilon,n}$. It follows from Lemma \ref{z-eta} that there exist $c'=c'(\varepsilon)\in (0,1)$ and $n_1=n_1(\varepsilon)>0$ (both depending only on $\beta$ and $\varepsilon$) such that for all $n\geq n_1$, we have $\mathds{P}[F_{\varepsilon, n}]\geq 1-\varepsilon/2$.
Moreover, according to \eqref{k0} and the definition of $b_k$ in \eqref{def-bk}, we can see that there exists a constant $M_1=M_1(\alpha)>0$ (depending only on $\beta$ and $\alpha$) such that for all $M\geq M_1$, we have
$
k_0=k_0(\alpha)=2.
$

We now apply $M=\max\{M_0, M_1\}$ and $L$ satisfying \eqref{cond-L} to the definition of $b_k$ in \eqref{def-bk}. Then combining this with $k_0=2$,  Lemma \ref{induct-thm} and  Proposition \ref{quantile-lb}, we can see that on the event $E_{1/2,n}\cap F_{\varepsilon,n}$,
\begin{equation*}
\begin{aligned}
\widehat{R}_n&\geq c\sum_{k=2}^{K_n}\frac{\min\{a_{\frac{1}{2},b_{k}+1},\cdots, a_{\frac{1}{2},b_{k-1}}\}}{d_k}\geq c\frac{\min\{a_{\frac{1}{2},b_2+1},\cdots, a_{\frac{1}{2},b_1}\}}{d_2}\\
&\geq \widetilde{c}(\varepsilon) \e^{\delta(\frac{1}{2})(n-2M-2L\log(2LM))}=:\widetilde{c}_1(\varepsilon)\e^{\delta(\frac{1}{2})n},
\end{aligned}
\end{equation*}
where $c$ and $\delta(\frac{1}{2})$ are parameters in Proposition \ref{quantile-lb} with $\alpha=1/2$ (depending only on $\beta$), and  $\widetilde{c}(\varepsilon),\widetilde{c}_1(\varepsilon)$ are positive constants (both depending only on $\beta$ and $\varepsilon$).
Therefore, from the above analysis we can find $\widetilde{c}_2(\varepsilon)>0$ (depending only on $\beta$ and $\varepsilon$) such that
$$
\mathds{P}\left[\widehat{R}_n\geq \widetilde{c}_2(\varepsilon)\e^{\delta(\frac{1}{2})n}\right]\geq \mathds{P}[E_{1/2,n}\cap F_{\varepsilon,n}] \geq 1-\varepsilon\quad \text{for all }n\geq \max\{n_0,n_1\}.
$$
Furthermore, since $n_0$ and $n_1$ depend only on $\beta$ and $\varepsilon$, we can take $\widetilde{c}_3(\varepsilon)>0$ (depending only on $\beta$ and $\varepsilon$) such that
$$
\mathds{P}\left[\widehat{R}_n\geq \widetilde{c}_3(\varepsilon)\e^{\delta(\frac{1}{2})n}\right]\geq 1-\varepsilon\quad \text{for all }n\geq 1.
$$
Hence, the proof is complete.
\end{proof}

\section{Proof of Theorem \ref{mr}}\label{sect-proof}

For $N\geq 1$, we recall that $R(0,[-N,N]^c)$ is the effective resistance between $0$ and $[-N,N]^c$ as defined in \eqref{eff-resis} with $I_1=\{0\}$ and $I_2=[-N,N]^c$, and that $\widehat{R}([-N,0],[-N,N]^c)$ is the effective resistance generated by unit flows from $[-N,0]$ to $[-N,N]^c$, passing through the interval $(0,N]$ as defined in \eqref{def-Rtilde}.
In particular, recall that $\widehat{R}((-\infty,0],(N,+\infty))=\widehat{R}_n$ since $N=2^n$.


As we mentioned in Subsection \ref{outline}, to complete the proof of Theorem \ref{mr}, it is essential to  establish a lower bound on  $\widehat{R}([-N,0],[-N,N]^c)$ in terms of $\widehat{R}_n$, which will allow us to use estimates for $\widehat{R}_n$  in Section \ref{sect-Rhat}.
To achieve this, note that flows in the definition of $\widehat{R}([-N,0],[-N,N]^c)$ will eventually flow into (either of) the two intervals $(-\infty, -N)$ and $(N,+\infty)$. Therefore, we further decompose $\widehat{R}([-N,0],[-N,N]^c)$ (see Figure \ref{split}) and define
\begin{equation}\label{Rtilde}
\begin{aligned}
\widetilde{R}_n&=\inf\left\{\frac{1}{2}\sum_{i\sim j }f^2_{ij}:\ f\ \text{is a unit flow from $[-N,0]$ to }(-\infty,-N),\right.\\
&\quad \quad \quad\quad \quad  \text{and }f_{ij}=0\ \text{for all }\langle i,j\rangle\in \mathcal{E}_{[-N,0]\times [-N,N]^c}\cup \mathcal{E}_{(0,N]\times (N,+\infty)} \Bigg\}
\end{aligned}
\end{equation}
as the effective resistance generated by unit flows from $[-N,0]$ to $(-\infty,-N)$, passing through the interval $(0,N]$.
Note that if there is no edge joining $(0,N]$ and  $(-\infty,-N)$, we have $\widetilde{R}_n=\infty$.

\begin{figure}[htbp]
\centering
\includegraphics[height=3.5cm,width=16cm]{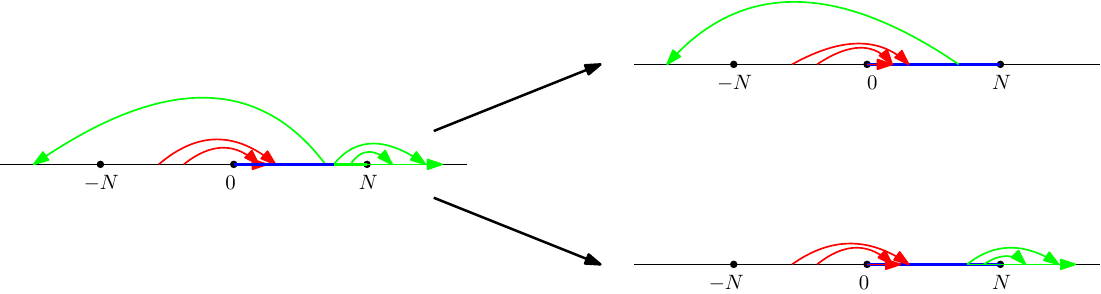}
\caption{The illustration for the ``decomposition'' of $\widehat{R}([-N,0],[-N,N]^c)$.
The red arrows represent flows entering $(0,N]$ from $[-N,0]$, while the green arrows represent flows exiting $(0,N]$ to $[-N,N]^c$. The blue lines represent the interval $(0,N]$ that flows must pass through.
Depending on whether flows ultimately enter the interval $(-\infty,-N)$ or $(N,+\infty)$, flows in the definition \eqref{def-Rtilde} of $\widehat{R}([-N,0],[-N,N]^c)$ can be divided into two portions corresponding to the two figures on the right.
The effective resistance generated by these flows in top right figure is $\widetilde{R}_n$, while the effective resistance generated by these flows in bottom right figure is $\widehat{R}([-N,0],(N,+\infty))$.}
\label{split}
\end{figure}

\subsection{Properties of $\widetilde{R}_n$}\label{propertyRhat}
In this subsection, our goal is to show the existence of a certain stochastic control between $\widetilde{R}_n$ and $\widehat{R}_n$ (see \eqref{Rtilde-Rhat} below). This will allow us to obtain that with high probability, $\widetilde{R}_n$ also exhibits an exponential lower bound from estimates for $\widehat{R}_n$ in Section \ref{sect-Rhat}.

Now for any $m\geq 1$, $i_1,\cdots i_m\in (0,N]$, and any $j_1,\cdots, j_m\in (-\infty,-N)$, it is obvious that
\begin{equation}\label{prop-edge}
\begin{aligned}
&\mathds{P}\left[\langle i_k,j_k \rangle\in \mathcal{E}_{(0,N]\times (-\infty,-N)}\ \text{for all }k\in [1,m]_\mathds{Z}\right]\\
&\leq  \mathds{P}\left[\langle i_k,|j_k| \rangle\in \mathcal{E}_{(0,N]\times (N,+\infty)}\ \text{for all }k\in [1,m]_\mathds{Z}\right].
\end{aligned}
\end{equation}
In addition,  assume that $\mathcal{E}_{[-N,N]^2}$ is given. According to the monotonicity property of the effective resistance $\widehat{R}(\cdot,\cdot)$ in Lemma \ref{mono} (2), we have that $\widehat{R}([-N,0],(N,\infty))$ is non-increasing with respect to the edge set connecting $(0,N]$ and $(N,+\infty)$. That is, for any two deterministic edge sets $E_1,E_2\subset (0,N]\times (N,+\infty)$ with $E_1\subset E_2$,
\begin{equation}\label{noninc-Rhat-1}
\widehat{R}([-N,0],(N,\infty))(E_2)\leq \widehat{R}([-N,0],(N,\infty))(E_1),
\end{equation}
where $\widehat{R}([-N,0],(N,\infty))(E_k),\ k=1,2$ represent effective resistances as defined in \eqref{def-Rhat} by replacing $\mathcal{E}_{(0,N]\times (N,+\infty)}$ by $E_k,\ k=1,2$, respectively.


From the above analysis, we claim that for any $x\geq 0$,
\begin{equation}\label{Rtilde-Rhat}
\mathds{P}\left[\widetilde{R}_n> x\right]\geq \mathds{P}\left[\widehat{R}([-N,0],(N,+\infty))> x \right]\geq \mathds{P}\left[\widehat{R}_n> x\right].
\end{equation}
Indeed, the second inequality can be obtained from the monotonicity property of the effective resistance in Lemma \ref{mono} (2).
For the first inequality in \eqref{Rtilde-Rhat}, let
$$
\widetilde{\mathcal{E}}_{(0,N]\times (N,+\infty)}=\{\langle i, |j|\rangle:\ \langle i,j\rangle \in \mathcal{E}_{(0,N]\times (-\infty,-N)}\}.
$$
From \eqref{prop-edge}, we can construct a coupling of $(\widetilde{\mathcal{E}}'_{(0,N]\times (N,+\infty)}, \mathcal{E}'_{(0,N]\times (N,+\infty)})$ such that
$$
\widetilde{\mathcal{E}}'_{(0,N]\times (N,+\infty)}\overset {\text{law}}{=}\widetilde{\mathcal{E}}_{(0,N]\times (N,+\infty)}\quad \text{and}\quad \mathcal{E}'_{(0,N]\times (N,+\infty)}\overset {\text{law}}{=}\mathcal{E}_{(0,N]\times (N,+\infty)},
$$
and $\widetilde{\mathcal{E}}'_{(0,N]\times (N,+\infty)}\subset \mathcal{E}'_{(0,N]\times (N,+\infty)}$.
Then combining this with \eqref{noninc-Rhat-1}, we arrive at
\begin{equation*}
\begin{aligned}
\widetilde{R}_n&\overset {\text{law}}{=}\widehat{R}_n([-N,0],(N,+\infty))(\widetilde{\mathcal{E}}'_{(0,N]\times (N,+\infty)})\\
&\geq \widehat{R}_n([-N,0],(N,+\infty))(\mathcal{E}'_{(0,N]\times (N,+\infty)})\overset {\text{law}}{=}\widehat{R}_n([-N,0],(N,+\infty)),
\end{aligned}
\end{equation*}
where $\widehat{R}_n([-N,0],(N,+\infty))(\widetilde{\mathcal{E}}'_{(0,N]\times (N,+\infty)})$ and $\widehat{R}_n([-N,0],(N,+\infty))(\mathcal{E}'_{(0,N]\times (N,+\infty)})$ are effective resistances as defined in \eqref{def-Rhat}
 by replacing $\mathcal{E}_{(0,N]\times (N,+\infty)}$ by $\widetilde{\mathcal{E}}'_{(0,N]\times (N,+\infty)}$ and $\mathcal{E}'_{(0,N]\times (N,+\infty)}$, respectively. This implies the first inequality in \eqref{Rtilde-Rhat}.

Combining  \eqref{Rtilde-Rhat} with Theorem \ref{Rn}, we get the following lemma.

\begin{lemma}\label{Rn-tilde}
For any $\beta>0$, there is a constant $\delta>0$ {\rm(}depending only on $\beta${\rm)} such that the following holds.
For any $\varepsilon\in(0,1/2]$, there exists a constant $c=c(\varepsilon)>0$ {\rm(}depending only on $\beta$ and $\varepsilon${\rm)} such that for all $n\geq 1$,
$$
\mathds{P}\left[\widetilde{R}_n\geq c{\rm e}^{\delta n}\right]\geq 1-\varepsilon.
$$
\end{lemma}

Furthermore, we have the following estimate for $\widehat{R}([-N,0],[-N,N]^c)$.
\begin{lemma}\label{lem-Rhat}
For any $\beta>0$, there is a constant $\delta>0$ {\rm(}depending only on $\beta${\rm)} such that the following holds.
For any $\varepsilon\in(0,1/2]$, there exists a constant $c=c(\varepsilon)>0$ {\rm(}depending only on $\beta$ and $\varepsilon${\rm)} such that for all $N=2^n\ (n\geq 1)$,
$$
\mathds{P}\left[\widehat{R}([-N,0],[-N,N]^c)\geq c{\rm e}^{\delta n}\right]\geq 1-\varepsilon.
$$
\end{lemma}
\begin{proof}
let $N=2^n$ for some $n\geq 1$. From Theorem \ref{Rn} and Lemma \ref{Rn-tilde}, there exists a constant $\delta>0$ (depending only on $\beta$) such that the following holds. For any $\varepsilon\in (0,1/2]$, there exists a constant $c=c(\varepsilon)>0$ (depending only on $\beta$ and $\varepsilon$) such that for all $n\geq 1$,
\begin{equation}\label{minR-1}
\mathds{P}[A]:=\mathds{P}\left[\min\left\{\widetilde{R}_n,\widehat{R}_n\right\}\geq c\e^{\delta n}\right]\geq 1-\varepsilon.
\end{equation}

Now assume that the event $A$ occurs. Let $f$ be the unit flow from $[-2^n,0]$ to $[-2^n,2^n]^c$, passing through the interval $(0,2^n]$, such that
$$
\widehat{R}([-2^n,0],[-2^n,2^n]^c)=\frac{1}{2}\sum_{i\sim j}f_{ij}^2.
$$
Denote $\theta_1$ and $\theta_2$ as portions of flow $f$ which enters intervals $(2^n,+\infty)$ and $(-\infty, -2^n)$, respectively. Then
$$
|\theta_1|=\sum_{i\in (0,2^n]}\sum_{j\in (2^n,+\infty)}f_{ij}\quad \text{and}
\quad |\theta_2|=\sum_{i\in (0,2^n]}\sum_{j\in (-\infty,-2^n)}f_{ij}.
$$
Clearly, $|\theta_1|+|\theta_2|=1$. In particular, if there is no edge joining $(0,2^n]$ and  $(-\infty,-2^n)$, then $\theta_2=0$.

Now according to the definition of $\widehat{R}(\cdot,\cdot)$ in \eqref{def-Rtilde}, the monotonicity property of $\widehat{R}$ in Lemma \ref{mono} (2), and the fact $|\theta_1|+|\theta_2|=1$, we obtain that on event $A$,
\begin{equation*}
\begin{aligned}
\widehat{R}([-2^n,0],[-2^n,2^n]^c)&\geq |\theta_1|^2\widehat{R}_n+ |\theta_2|^2\widetilde{R}_n\geq (|\theta_1|^2+|\theta_2|^2)c\e^{\delta n}\geq c\e^{\delta n}/2.
\end{aligned}
\end{equation*}
Combining this with \eqref{minR-1}, we complete the proof.
\end{proof}

\subsection{Proof of Theorem \ref{mr}}

The proof of Theorem \ref{mr} is divided into proofs of \eqref{point} and \eqref{box}. Since there are similarities between the two proofs, we will omit some details of the similar parts in the proof of \eqref{box}. Let us start by the

\begin{proof}[Proof of  \eqref{point} in Theorem \ref{mr}]
We begin by considering the case $N=2^n$ for some $n\geq 1$. According to the translation invariance of the model, we can see that
$\widehat{R}([-N,0],[-N,N]^c)$ and $\widehat{R}([0,N],[-N,N]^c)$ have the same distribution.
Therefore, from Lemma \ref{lem-Rhat}, there exists a constant $\delta>0$ (depending only on $\beta$) such that the following holds. For any $\varepsilon\in (0,1/2]$, there exists a constant $c=c(\varepsilon)>0$ (depending only on $\beta$ and $\varepsilon$) such that for all $n\geq 1$,
\begin{equation}\label{minR}
\mathds{P}[B]:=\mathds{P}\left[\min\left\{\widehat{R}([-N,0],[-N,N]^c),\widehat{R}([0,N],[-N,N]^c)\right\}\geq c\e^{\delta n}\right]\geq 1-\varepsilon.
\end{equation}

Now assume that the event $B$ occurs.  Let $f$ be the unit flow from 0 to $[-N,N]^c$ that minimizes the right-hand side of \eqref{eff-resis}, that is,
\begin{equation*}
R(0,[-N,N]^c)=\frac{1}{2}\sum_{i\sim j}f^2_{ij}.
\end{equation*}
Since $f$ must flow into $[-N,N]^c$ finally, it must pass through either $[-N,0]$ or $[0,N]$. 
In light of this, let $\theta_1$ be the portion of flow $f$ that passes through the interval $(0,N]$ and then flow into $[-N,N]^c$.
Similarly, we define $\theta_2$ by replacing $(0,N]$ with $[-N,0)$ in the definition of $\theta_1$.
Then we have $\max\{|\theta_1|,|\theta_2|\}\geq 1/2$ (see \eqref{flow+-} for more details).
Combining this with the definition of the effective resistance $\widehat{R}(\cdot,\cdot)$ in \eqref{def-Rtilde}, we get that on the event $B$,
\begin{equation}\label{R2m-lb}
\begin{aligned}
R(0,[-N,N]^c)&\geq \max\left\{|\theta_1|^2\widehat{R}([-N,0],[-N,N]^c), |\theta_2|^2\widehat{R}([0,N],[-N,N]^c)\right\}\\
&\geq  \max\left\{|\theta_1|^2, |\theta_2|^2\right\}c\e^{\delta n}\geq \frac{c}{4}\e^{\delta n}=:c_1N^{\delta'}.
\end{aligned}
\end{equation}
Here $c_1=c/4$ is a constant depending only on $\beta$ and $\varepsilon$, and $\delta'=\delta/\log2$ is a constant depending only on $\beta$.
 Hence, from \eqref{minR} and \eqref{R2m-lb} we can conclude
\begin{equation}\label{2m}
\mathds{P}\left[R(0,[-N,N]^c)\geq c_1N^{\delta'}\right] \geq \mathds{P}[B]\geq 1-\varepsilon.
\end{equation}

For general $N\geq 1$, let $n=\lfloor \log_2 N\rfloor$ such that $N=2^{n+p}$ with $p\in [0,1)$. Since the effective resistance $R(0,[-r,r]^c)$ is nondecreasing with $r$, we have
$$
R(0,[-N,N]^c)\geq R(0,[-2^n,2^n]^c).
$$
Therefore, for any $\varepsilon \in (0,1/2]$, by \eqref{2m} we obtain that
$$
\mathds{P}\left[R(0,[-N,N]^c)\geq c_2 N^{\delta'} \right]\geq \mathds{P}\left[R(0,[-2^n,2^n]^c)\geq c_2N^{\delta'}\right]\geq 1-\varepsilon,
$$
where $c_2:=c_12^{-\delta'}$ is a constant depending only on $\beta$ and $\varepsilon$. Hence, the proof is complete.
\end{proof}

We next turn to the effective resistance $R([-N,N],[-2N,2N]^c)$.  We also first consider the case $N=2^n$ for some $n\geq 1$. Similar to \eqref{Rtilde}, we define
\begin{equation*}
\begin{aligned}
\widetilde{R}([-2N,N],(-\infty,-2N))&=\inf\left\{\frac{1}{2}\sum_{i\sim j}f^2_{ij}:\ f\text{ is a unit flow from $[-2N,N]$ to }(-\infty,-2N),\right.\\
&\quad \quad  \text{and $f_{ij}=0$ for all }\langle i,j\rangle\in \mathcal{E}_{[-2N,N]\times[-2N,2N]^c}\cup \mathcal{E}_{(N,2N]\times (2N,+\infty)}\Bigg\}
\end{aligned}
\end{equation*}
as the effective resistance generated by unit flows from $[-2N,N]$ to $(-\infty,-2N)$, passing through the interval $(N,2N]$. Note that if there is no edge joining $(N,2N]$ and $(-\infty,-2N)$, we then have $\widetilde{R}([-2N,N],(-\infty,-2N))=\infty$. Using similar arguments for \eqref{Rtilde-Rhat} and the translation invariance of the model, we have that for all $x\geq 0$,
\begin{equation*}
\begin{aligned}
\mathds{P}\left[\widetilde{R}([-2N,N],(-\infty,-2N))\geq x\right]&\geq \mathds{P}\left[\widehat{R}([-2N,N],(2N,+\infty))\geq x\right]\\
&\geq\mathds{P}\left[ \widehat{R}([-\infty,N],(2N,+\infty))\geq x\right]=\mathds{P}\left[\widehat{R}_n\geq x\right].
\end{aligned}
\end{equation*}
Combining this with similar arguments in the proof of Lemma \ref{lem-Rhat}, we obtain the following estimate for $\widehat{R}([-2N,N],[-2N,2N]^c)$, which is defined in \eqref{def-Rtilde} by replacing $[-N,0]$ and $[-N,N]^c$  with $[-2N,N]$ and $[-2N,2N]^c$, respectively.
\begin{lemma}\label{lem-Rhat-box}
For any $\beta>0$, there is a constant $\delta>0$ {\rm(}depending only on $\beta${\rm)} such that the following holds. For any $\varepsilon\in (0,1/2]$, there exists a constant $c=c(\varepsilon)>0$  {\rm(}depending only on $\beta$ and $\varepsilon${\rm)} such that for all $N=2^n\ (n\geq 1)$,
$$
\mathds{P}\left[\widehat{R}([-2N,N],[-2N,2N]^c)\geq c\e^{\delta n}\right]\geq 1-\varepsilon.
$$
\end{lemma}

We now can present the

\begin{proof}[Proof of  \eqref{box} in Theorem \ref{mr}]
Throughout the proof, we condition on $\mathcal{E}_{[-N,N]\times [-2N,2N]^c}=\varnothing$, i.e., intervals $[-N,N]$ and $[-2N,2N]^c$ are not directly connected by any long edge.

We begin by considering the case $N=2^n$ for some $n\geq 1$. It is clear from the translation invariance of the model that $\widehat{R}([-2N,N],[-2N,2N]^c)$ and $\widehat{R}([-N,2N],[-2N,2N]^c)$ have the same distribution. Hence, from Lemma \ref{lem-Rhat-box}, there exists a constant $\delta>0$ (depending only on $\beta$) such that the following holds. For any $\varepsilon\in (0,1/2]$, there exists a constant $c=c(\varepsilon)>0$ (depending only on $\beta$ and $\varepsilon$) such that for all $n\geq 1$,
\begin{equation}\label{minR-2}
\mathds{P}[C]:=\mathds{P}\left[\min\left\{\widehat{R}([-2N,N],[-2N,2N]^c),\widehat{R}([-N,2N],[-2N,2N]^c)\right\}\geq c\e^{\delta n}\right]\geq 1-\varepsilon.
\end{equation}
Note that, by the definition of $\widehat{R}(\cdot,\cdot)$, we can observe that the two effective resistances in \eqref{minR-2} are both independent of the edge set $\mathcal{E}_{[-N,N]\times [-2N,2N]^c}$. Therefore, \eqref{minR-2} implies that
\begin{equation}\label{minR-cond}
\mathds{P}[C\big|\mathcal{E}_{[-N,N]\times [-2N,2N]^c}=\varnothing]=\mathds{P}[C]\geq 1-\varepsilon.
\end{equation}

Now assume that the event $C$ occurs. Let $\widetilde{f}$ be the unit flow from $[-N,N]$ to $[-2N,2N]^c$ such that
$$
R([-N,N],[-2N,2N]^c)=\frac{1}{2}\sum_{i\sim j} \widetilde{f}_{ij}^2.
$$
Similar to the argument in the proof of  \eqref{point}, we can define $\widetilde{\theta}_1$ as the portion of flow $\widetilde{f}$ that passes through the interval $(N,2N]$ and then flow into $[-2N,2N]^c$, and define $\widetilde{\theta}_2$ by replacing $(N,2N]$ with $[-2N,-N)$ in the definition of $\widetilde{\theta}_1$. Then we also have $\max\{|\widetilde{\theta}_1|, |\widetilde{\theta}_2|\}\geq 1/2$. Therefore, we get that on the event $C$,
$$
\begin{aligned}
R([-N,N],[-2N,2N]^c)&\geq \max\left\{|\widetilde{\theta}_1|^2\widehat{R}([-2N,N],[-2N,2N]^c),\ |\widetilde{\theta}_2|^2\widehat{R}([-N,2N],[-2N,2N]^c)\right\}\\
&\geq \max\left\{|\widetilde{\theta}_1|^2,\ |\widetilde{\theta}_2|^2\right\}c\e^{\delta n}\geq \frac{c}{4}\e^{\delta n}:=c_1N^{\delta'}.
\end{aligned}
$$
Here $c_1=c/4$ is a constant depending only on $\beta$ and $\varepsilon$, and $\delta'=\delta/\log 2$ is a constant depending only on $\beta$. Hence, combining this with \eqref{minR-cond}, we can conclude
$$
\mathds{P}\left[R([-N,N],[-2N,2N]^c)\geq c_1N^{\delta'}\big|\mathcal{E}_{[-N,N]\times [-2N,2N]^c}=\varnothing\right]\geq 1-\varepsilon.
$$

For general $N\geq 1$, we can complete the proof by using the similar argument in the proof of \eqref{point}.
\end{proof}

\bigskip

\noindent{\bf Acknowledgement.} \rm
We warmly thank Jian Wang for stimulating discussions at an early stage of the project. L.-J. Huang would like to thank School of Mathematical Sciences, Peking University for their hospitality during
her visit. J.\ Ding is partially supported by NSFC Key Program Project No.\ 12231002 and the Xplorer prize. L.-J. Huang is partially supported by National Key R\&D Program of China No.\ 2023YFA1010400.

\bibliographystyle{plain}
\bibliography{resistance-ref}

\end{document}